\documentclass[12pt]{amsart}
\usepackage{amssymb,latexsym,amsmath,amsthm,amscd}
\usepackage{setspace}
\usepackage[active]{srcltx}
\usepackage[all]{xy} \xyoption{arc}
\usepackage[left=3cm,top=2cm,right=3cm,bottom = 2cm]{geometry}
\usepackage{graphicx}
\usepackage[usenames,dvipsnames]{color}
\usepackage{charter}
\usepackage{hyperref}

\theoremstyle{plain}
\newtheorem{thm}{Theorem}
\newtheorem{lem}[thm]{Lemma}
\newtheorem{prop}[thm]{Proposition}
\newtheorem{cor}[thm]{Corollary}

\theoremstyle{definition}
\newtheorem{dfn}[thm]{Definition}
\newtheorem{ex}[thm]{Example}

\theoremstyle{remark}
\newtheorem{rmk}[thm]{Remark}

\newcommand{\cF}{\mathcal{F}}

\newcommand{\cN}{\mathcal{N}}
\newcommand{\cO}{\mathcal{O}}

\newcommand{\cV}{\mathcal{V}}

\newcommand{\veps}{\varepsilon}

\DeclareMathOperator{\uhp}{\mathcal{H}}

\DeclareMathOperator{\Tr}{Tr}

\DeclareMathOperator{\Hom}{Hom}

\DeclareMathOperator{\GL}{GL}

\DeclareMathOperator{\SL}{SL}
\DeclareMathOperator{\PSL}{PSL}

\DeclareMathOperator{\Sym}{Sym}

\DeclareMathOperator{\Res}{Res}

\DeclareMathOperator{\Ind}{Ind}

\DeclareMathOperator{\diag}{diag}

\newcommand*{\df}{\mathrel{\vcenter{\baselineskip0.5ex \lineskiplimit0pt
                     \hbox{\scriptsize.}\hbox{\scriptsize.}}} =}

\providecommand{\twomat}[4]{\left(\begin{matrix}#1&#2\\#3&#4\end{matrix}\right)}
\providecommand{\stwomat}[4]{\left(\begin{smallmatrix}#1&#2\\#3&#4\end{smallmatrix}\right)}
\providecommand{\twovec}[2]{\left(\begin{matrix}#1\\#2\end{matrix}\right)}

\newcommand{\CC}{\mathbf{C}}

\newcommand{\ZZ}{\mathbf{Z}}
\newcommand{\PP}{\mathbf{P}}
\newcommand{\RR}{\mathbf{R}}

\DeclareMathOperator{\GM}{GM}
\DeclareMathOperator{\tpexp}{\otimes_e}

\begin{document}

\title[Constructions of vvmfs]{Constructions of vector-valued modular forms of rank four and level one}
\author{Cameron Franc}
\email{franc@math.usask.ca}
\author{Geoff Mason}
\email{gem@ucsc.edu}
\thanks{The first author was partially supported by NSERC grant RGPIN-2017-06156. The second author was supported by the Simons Foundation $\# 427007$.}

\begin{abstract}
This paper studies modular forms of rank four and level one. There are two possiblities for the isomorphism type of the space of modular forms that can arise from an irreducible representation of the modular group of rank four, and we describe when each case occurs for general choices of exponents for the $T$-matrix. In the remaining sections we describe how to write down corresponding differential equations satisfied by minimal weight forms, and how to use these minimal weight forms to describe the entire graded module of holomorphic modular forms. Unfortunately the differential equations that arise can only be solved recursively in general. We conclude the paper by studying the cases of tensor products of two-dimensional representations, symmetric cubes of two-dimensional representations, and inductions of two-dimensional representations of the subgroup of the modular group of index two. In these cases the differential equations satisfied by minimal weight forms can be solved exactly.
\end{abstract}
\maketitle

\tableofcontents
\section{Introduction}
\label{s:intro}
In this paper we study methods for describing holomorphic modular forms that transform according to four dimensional complex representations of the group $\Gamma = \SL_2(\ZZ)$. In ranks two and three there is a relatively complete picture -- see for example \cite{FrancMason1}, \cite{FrancMason2} and \cite{FrancMason3}. These results have been used to prove instances of the unbounded denominator conjecture \cite{ASD}, and to classify VOAs according to the monodromy of the associated vector-valued modular form \cite{MasonNagatomoSakai}. Unfortunately the situation becomes more complicated for representations of rank four and higher, and one does not have as complete control as one would like.

There are two main results in this paper:
\begin{enumerate}
\item[(a)] in each of the two cases that can arise in rank four, we explain in Sections \ref{s:cyclic} and \ref{s:noncyclic} how to recursively compute a free basis of modular forms for any irreducible representation of $\Gamma$ of rank four;
\item[(b)] in Sections \ref{s:tensorproduct}, \ref{s:symmetriccube} and \ref{s:induction} we explain how to obtain exact formulas for bases of modular forms for irreducible representations of rank four that are obtained via linear algebraic constructions from representations of rank two.
\end{enumerate}
The solution to problem (a) amounts to determining the system of ordinary differential equations satisfied by the minimal weight form for a given representation, and then showing how the solution to this system of equations can be used to produce a free basis of modular forms.

In problem (b) a new issue arises: the space of holomorphic forms for a given representation $\rho$ corresponds to a canonical lattice $M(\rho,L)$ inside the space $M^\dagger(\rho)$ of weakly holomorphic modular forms for $\rho$. Here $L$ denotes an exponent matrix satisfying $\rho(T) = e^{2\pi i L}$ where $T = \stwomat 1101$, and $M(\rho,L)$ is the corresponding space of modular forms whose behaviour at the cusp is determined by $L$. See Section \ref{s:vvmfs} for a detailed discussion. We would like to be able to describe, concretely in terms of classical functions, a basis for $M(\rho,L)$. Instead, in each of the cases discussed in Sections \ref{s:tensorproduct}, \ref{s:symmetriccube} and \ref{s:induction}, we are able to produce an explicit basis for a larger space $M(\rho,L')$ of modular forms for some functorially induced choice of exponents $L'$. It is then possible to determine the subspace $M(\rho,L) \subseteq M(\rho,L')$ using linear algebra by examining $q$-expansions. This is similar to computing a space of modular forms of weight one by computing a larger space of modular forms of weight two, and then dividing the appropriate subspace by the square $\eta^2$ of the Dedekind eta function to recover the desired space of forms of weight one.

To provide a bit more detail, suppose for the sake of definiteness that $\rho = \Sym^3\alpha$ for some two-dimensional representation $\alpha$ of $\Gamma$, and let $L$ be a choice of exponents for $\alpha$ (not necessarily canonical). If $(f,g)^t$ is a minimal weight form for $\alpha$, then we describe in Section \ref{s:symmetriccube} the corresponding functorial choice of exponents $\Sym^3L$ such that $(f^3,f^2g,fg^2,g^3)^t$ is a minimal weight form in $M(\Sym^3\alpha,\Sym^3L)$. Combined with the solution to problem (a), this allows one to write exact formulas for bases of spaces of modular forms of the form $M(\Sym^3\alpha,\Sym^3L)$. As the exponents $L$ vary, the lattices $M(\Sym^3\alpha,\Sym^3L)$ are cofinal in the space $M^\dagger(\Sym^3\alpha)$ of all weakly holomorphic modular forms for $\Sym^3\alpha$. Thus, using our results, one can compute formulas for any weakly holomorphic modular form for an irreducible representation of rank four of the form $\Sym^3\alpha$. Section \ref{s:tensorproduct} discusses the case of tensor products, and Section \ref{s:induction} discusses the case of induction of representations of rank two of the subgroup $G \subseteq \Gamma$ of index two. Section \ref{s:induction} makes use of the results from \cite{CandeloriFranc2} on vector valued modular forms for $G$.

The moduli space of irreducible representations of $\Gamma$ of rank four is three-dimensional, while the families of such representations that arise by tensor product are two-dimensional, and the families that arise from symmetric cubes and induction are one-dimensional. Thus, our solution to problem (b) only covers a small portion of all moduli of representations of rank four. Nevertheless, for these families of representations one could prove new instances of the unbounded denominator conjecture as in \cite{FrancMason1}, \cite{FrancMason2} and \cite{FrancMason3}. The next step would be to classify the subset of representations that are of finite image, and then to classify which of those are congruence.

At this point we should mention the recent and interesting paper \cite{WesterholtRaum} of Westerholt-Raum, which studies the hyperalgebra structures that arise when considering tensor products of vector-valued modular forms. One of the aims of \cite{WesterholtRaum} is to establish results concerning generation of spaces of modular forms by products of Eisenstein series, and for this reason the paper \cite{WesterholtRaum} focuses on the case of congruence representations of $\Gamma$. In this paper we do not assume that the image of $\rho$ is finite.

\section{Vector-valued modular forms}
\label{s:vvmfs}
Let $\rho \colon \Gamma \to \GL_d(\CC)$ denote a representation of $\Gamma = \SL_2(\ZZ)$, and define as usual
\begin{align*}
  T &= \twomat 1101, & S&= \twomat 0{-1}10, & R&= ST=\twomat 0 {-1} 11.
\end{align*}
\begin{dfn}
  A \emph{choice of exponents} for $\rho$ is a matrix $L$ such that $\rho(T) = e^{2\pi i L}$.
\end{dfn}
Since the matrix exponential is surjective, there always exist choices of exponents for any representation.
\begin{dfn}
\label{d:wvvmfs}
  A function $F \colon \uhp \to \CC^d$ is said to be a \emph{weakly holomorphic modular form} for $\rho$ of weight $k \in \ZZ$ provided that the following conditions are satisfied:
  \begin{enumerate}
  \item $F$ is holomorphic;
  \item for all $\gamma = \stwomat abcd \in \Gamma$,
    \[
  F(\gamma \tau) = (c\tau+d)^k\rho(\gamma)F(\tau);
\]
\item for one choice of exponents $L$, the function $\tilde F(\tau) = e^{-2\pi i L\tau}F(\tau)$ has a meromorphic $q$-expansion, where $q = e^{2\pi i \tau}$.
\end{enumerate}
\end{dfn}

If $F$ is a weakly holomorphic modular form for $\rho$, then by condition (2) of Definition \ref{d:wvvmfs}, $\tilde F$ satisfies $\tilde F(\tau+1) = \tilde F(\tau)$. Since $F$ is holomorphic on $\uhp$, a standard argument then shows that $\tilde F$ has a convergent Laurent expansion in $q$. The meromorphy hypothesis of (3) in Defintion \ref{d:wvvmfs} is independent of the choice of exponents $L$. Let $M^\dagger_k(\rho)$ denote the set of all weakly holomorphic modular forms of weight $k$ for $\rho$, and set $M^\dagger(\rho) = \bigoplus_{k \in \ZZ} M_k^\dagger(\rho)$.

\begin{ex}
If $\rho$ is the trivial representation, then $M_0^\dagger(\rho) = \CC[j]$ where $j$ denotes the classical $j$-function, and $M^\dagger(\rho) = M(1)[1/\Delta]$.
\end{ex}

\begin{dfn}
\label{d:wvvmfs}
  Fix a choice of exponents $L$ for $\rho$. A function $F \colon \uhp \to \CC^d$ is said to be a \emph{holomorphic modular form} for $(\rho,L)$ of weight $k \in \ZZ$ provided that the following conditions are satisfied:
  \begin{enumerate}
  \item $F$ is holomorphic;
  \item for all $\gamma = \stwomat abcd \in \Gamma$,
    \[
  F(\gamma \tau) = (c\tau+d)^k\rho(\gamma)F(\tau);
\]
\item the function $\tilde F(\tau) = e^{-2\pi i L\tau}F(\tau)$ has a holomorphic $q$-expansion, where $q = e^{2\pi i \tau}$.
\end{enumerate}
\end{dfn}
Let $M_{k}(\rho,L)$ denote the set of holomorphic modular forms for $(\rho,L)$ of weight $k$, and let $M(\rho,L) = \bigoplus_{k \in \ZZ}M_k(\rho,L)$. The possible choices of exponents $L$ are classified by the real parts of their eigenvalues. If $L$ and $L'$ are two different choices of exponents for a fixed $\rho$, then the real parts of their eigenvalues differ by integers. Thus for a fixed $\rho$, the set of exponent matrices for $\rho(T)$ forms a lattice of rank equal to the number of Jordan blocks in the Jordan decomposition of $\rho(T)$. One can define a lexicographic partial ordering on the set of exponents, and in terms of this lexicographic ordering, for all $L \leq L'$ there is an order reversing inclusion:
\[
  M(\rho,L') \subseteq M(\rho,L) \subseteq M^\dagger(\rho).
\]
Each space $M(\rho,L)$ has the structure of a graded module over the ring $M(1) = \CC[E_4,E_6]$ of modular forms of level one, and $M^\dagger(\rho)$ has the structure of a graded module over $M(1)[1/\Delta]$. These structures are compatible with the inclusions $M(\rho,L) \subseteq M^\dagger(\rho)$.
\begin{ex}
  If the real parts of the eigenvalues of $L$ are chosen to lie in $[0,1)$, while $L'$ has all real parts chosen to lie in $(0,1]$, then $M(\rho,L)$ is the space of holomorphic modular forms with the standard definition of holomorphy at the cusp, and $M(\rho,L')$ is the subspace of cusp forms. In \cite{CandeloriFranc1}, $L$ was refered to as the canonical choice of exponents, following \cite{Deligne}.
\end{ex}

In \cite{CandeloriFranc1} it was explained how the spaces $M_k(\rho,L)$ can be interpreted as spaces of global sections of certain vector bundles on the  moduli space of elliptic curve. The interpretation is by now a quite classical part of the Riemann-Hilbert correspondence -- see \cite{Simpson} and \cite{Deligne} for discussions of this correspondence in a language that is quite close to the language of modular forms introduced above. Let us recall how this correspondence works.

First, the Riemann-Hilbert correspondence most naturally associates to $\rho$ a local system, equivalently, a flat holomorphic connection, on the open modular curve $Y(1) = \Gamma \backslash \uhp$ of level $1$. This bundle can be described concretely by pulling back along the uniformizing map $\uhp \to Y(1)$ as a trivial bundle $\uhp \times \CC^d$ endowed with an action of $\Gamma$:
\[
  \gamma (\tau,v) = (\gamma \tau, \rho(\gamma) v).
\]
Write $\cV(\rho)$ for this bundle. Its meromorphic extension to the cusp (see Section 0.8 of \cite{Sabbah} for a nice discussion of meromorphic bundles) will be denoted $\cV^\dagger(\rho)$, and $M^\dagger_0(\rho)$ is the space of global sections of $\cV^\dagger(\rho)$. More generally, the spaces $M^\dagger_k(\rho)$ are global sections of twists $\cV^\dagger_k(\rho) \df \cV^\dagger(\rho)\otimes \cO(k)$.

If $L$ is a choice of exponents for $\rho(T)$, then \cite{CandeloriFranc1} (see also page 738 of \cite{Simpson}) explains how to interpret $M_k(\rho,L)$ as the space of global sections of a vector bundle $\cV_k(\rho,L) \subseteq \cV^\dagger_k(\rho)$. In the terminology of \cite{Sabbah}, the various bundles $\cV_k(\rho,L)$ are examples of \emph{lattices} inside the meromorphic bundle $\cV^\dagger(\rho)$.
\begin{rmk}
  It is explained in great detail in \cite{Simpson} that the set of pairs $(\rho,L)$ is \emph{not} the right set of objects for considering the Riemann-Hilbert correspondence. Rather, one should consider filtered representations $(\rho,F)$ defined as follows: if $\rho$ is a representation with underlying vector space $V$, then $F$ is a decreasing filtration on $V$ indexed by real numbers, such that $F$ is left continuous, and for all $r \in \RR$ one has $\rho(T)F_rV \subseteq F_rV$. Filtered representations form a category, and the Riemann-Hilbert correspondence identifies this category with corresponding categories of filtered local systems, or equivalently, with the category of filtered holomorphic connections on $Y(1)$ with a regular singularity at the cusp. See \cite{Simpson} for details.

  The crux of the matter for us is that some lattices inside $\cV^\dagger_k(\rho)$ are not of the form $\cV_{k}(\rho,L)$ for a choice of exponents $L$ (see example \ref{ex:missinglattice} below). This is not an issue for the present paper, however, as all lattices in $\cV^\dagger_k(\rho)$ are contained in \emph{some} $\cV_k(\rho,L)$, and the most important spaces of modular forms (such as classical holomorphic forms and cusp forms) are spaces of global sections of lattices $\cV_k(\rho,L)$. Thus in this paper we will only work with spaces of modular forms associated to pairs $(\rho,L)$, rather than those coming from all filtered representations of $\Gamma$.
\end{rmk}

Using the geometric interpretation of the spaces $M_k(\rho,L)$ and the splitting principle for vector bundles on the compact modular curve $X(1)$, one obtains an easy proof of the following result, which generalizes a result of \cite{MarksMason}:
\begin{thm}
  \label{t:fmt}
  Let $(\rho,L)$ denote a representation $\rho$ of $\Gamma$, and a corresponding choice of exponents $L$. Then:
  \begin{enumerate}
  \item The module $M(\rho,L)$ of holomorphic modular forms for $(\rho,L)$ is free of rank $\dim \rho$ over the ring $M(1)$ of modular forms of level one;
  \item The module $M^\dagger(\rho)$ is a free $M(1)[1/\Delta]$-module of rank $\dim \rho$.
  \end{enumerate}
\end{thm}
\begin{proof}
  A statement equivalent to the result above was first proved by Gannon in Theorem 3.3 of \cite{Gannon} under an admissibility hypothesis. See \cite{CandeloriFranc1} for a proof of (1) without any admissibility hypotheses. It remains to deduce (2) from (1) without the admissibility hypothesis of \cite{Gannon}. The point is that if $L$ denotes any choice of exponents, then the natural localisation map
\[
  \phi \colon M(\rho,L) \otimes_{M(1)} M(1)[1/\Delta] \ \to M^\dagger(\rho)
\]
is an isomorphism of graded modules (this justifies calling $M(\rho,L)$ a \emph{lattice} in $M^\dagger(\rho)$). The injectivity of $\phi$ follows from part (1). To see that the map is surjective, observe that if $F \in M^\dagger_k(\rho)$ then some multiple $\Delta^nF$ is contained in $M_{k+12n}(\rho,L)$, so that $F = \phi(\Delta^nF\otimes \Delta^{-n})$.
\end{proof}
\begin{rmk}
Part (1) of Theorem \ref{t:fmt} holds more generally for the space of modular forms associated to any filtered representation $(\rho,F)$ by the same argument as in \cite{CandeloriFranc1}.
\end{rmk}
\begin{rmk}
The weights of a basis of modular forms in $M(\rho,L)$ is a unique invariant of $(\rho,L)$, corresponding to the fact that vector bundles on the moduli space of elliptic curves decompose uniquely into line bundles. The weights of a free basis for $M^\dagger(\rho)$ are \emph{not} unique in general.
\end{rmk}

The bundles $\cV(\rho,L) = \cV_0(\rho,L)$ are endowed with natural regular connections, which are the restrictions of a natural regular connection on the meromorphic bundle $\cV^\dagger(\rho)$. In the natural description of $\cV(\rho,L)$ as a trivial bundle on $\uhp$, the connection is the usual holomorphic derivative. As a regular connection on the compact moduli space $X$ it is a map
\[
  \nabla \colon \cV(\rho,L) \to \cV(\rho,L) \otimes \Omega^1_X(\infty)
\]
where $\Omega^1_X(\infty)$ is the bundle of regular differentials with simple (equivalently, logarithmic) poles at the cusp $\infty$ of the compact modular curve $X$. Since $\Omega^1_X(\infty) \cong \cO(2)$, the connection defines a map $\cV_0(\rho,L) \to \cV_2(\rho,L)$. At the level of global sections, this is the usual map
\[
\frac{d}{d\tau}\colon M_0(\rho,L) \to M_2(\rho,L)
\]
taking a modular form of weight $0$ to a modular form of weight $2$.

The logarithmic connection $(\cV(\rho,L),\nabla)$ has an associated \emph{residue}, defined in terms of the action of $d/d\tau$ on a basis of local flat sections in an angular neighbourhood of the cusp. To be concrete, take for the angular neighbourhood the region
\[
U = \{x+iy \mid 0 < x < 1,~ y > 2\}
\]
in the upper half plane. If $v_1,\ldots, v_d$ is a basis for the vector space underlying $\rho$, then the functions $\tilde v_j \df e^{-2\pi i L\tau}v_j$ define a single-valued frame in the neighbourhood of the cusp determined by $U$. Observe that $\frac{d}{d\tau} \tilde v_j = -2\pi i L\tilde v_j$. The standard convention for residues is to divide by the period $2\pi i$ and define
\[
  \Res_\infty(\cV(\rho,L),\nabla) \df -L.
\]

Finally, below we will need to make use of the higher weight modular derivatives
\[
  D_k \colon M_k(\rho,L) \to M_{k+2}(\rho,L)
\]
defined by setting $D_k(F) = \frac{dF}{d\tau}-\frac{k}{12}E_2F$ where $E_2$ is the usual quasi-modular Eisenstein series of weight $2$, whose constant term is normalized to equal $1$. These operators arise from the natural connection on the $k$th symmetric power of the Gauss-Manin connection associated to the moduli space $X$. For details, see the appendix of \cite{Katz2} and tensor the construction described there with $(\cV(\rho,L),\nabla)$. If $D \colon M(\rho,L) \to M(\rho,L)$ denotes the corresponding graded operator that increase weights by $2$ ,then this makes $M(\rho,L)$ a graded module over the noncommutative ring $M(1)\langle D\rangle$ of modular linear differential operators.

\begin{ex}
  \label{ex:missinglattice}
  Let $\rho$ denote the inclusion representation, so that the exponents for $\rho(T)$ are of the form
  \[
  2\pi i L_n = \twomat {2\pi i n}10{2\pi i n},
\]
for $n \in \ZZ$. The canonical choice of exponents corresponds to the matrix $L_0$, and the corresponding connection $(\cV(\rho,L_0),\nabla)$ is isomorphic with the Gauss-Manin connection associated to the moduli space $X$. It is well-known (see e.g. the appendix of \cite{Katz2}) that there is an exact Hodge sequence
\[
  0 \to \cO(1) \to \cV(\rho,L_0) \to \cO(-1) \to 0.
\]
Theorem \ref{t:fmt} follows from the fact that this sequence splits $\cV(\rho,L_0) \cong \cO(1) \oplus \cO(-1)$. At the level of modular forms, a basis for $M(\rho,L_0)$ is described as follows: the natural homology bases of the elliptic curves $\CC/(\ZZ \oplus \ZZ\tau)$ define a modular form
\[
  F(\tau)= \twovec \tau 1 \in M_{-1}(\rho,L_0).
\]
This is the unique, up to rescaling, form of minimal weight in $M(\rho,L_0)$, and $F,DF$ defines a basis for $M(\rho,L_0)$ over $M(1)$. More generally,
\[
  M(\rho,L_n) = \Delta^nM(\rho,L_0) = M(1)\Delta^nF \oplus M(1) \Delta^n DF,
\]
where $\Delta$ is the usual Ramanujan $\Delta$-function satisfying $D_{12}\Delta = 0$.

Observe that $M(\rho,L_1) \subseteq M(\rho,L_0)$ is the subspace of cusp forms. Define $N \subseteq M(\rho,L_0)$ to be the subset of forms whose constant term is proportional to $(1,0)^T$. Since $M(\rho,L_1) \subseteq N \subseteq M(\rho,L_0)$, $N$ defines a lattice in $M^\dagger(\rho)$, in the sense that $N\otimes_{M(1)}M(1)[1/\Delta] \cong M^\dagger(\rho)$. Further, $N$ is a free $M(1)$-module with generators in weights $5$ and $7$. This module corresponds to a vector bundle $\cN$ of rank $2$ on $X$ isomorphic with $\cO(-5)\oplus \cO(-7)$. There are inclusions
\[
  \cV(\rho,L_1) \subsetneq \cN \subsetneq \cV(\rho,L_0)
\]
that induce isomorphisms away from the cusps, but these vector bundles are not isomorphic at the cusp. Thus $\cN$ gives an example of a bundle that is not of the form $\cV(\rho,L)$. Instead, it corresponds to a filtered representation in the sense of \cite{Simpson}.
\end{ex}

\begin{rmk}
Most of the discussion above applies more generally to any Fuchsian group\footnote{One exception is that in general the Free-Module Theorem \ref{t:fmt} does not hold -- see \cite{CandeloriFranc2} for a discussion of this point.}. Instead of choosing a single exponent matrix at $\infty$, one makes a choice of exponents for each cusp of the corresponding modular curve. We will require this more general setting when we discuss induction in Section \ref{s:induction} below.
\end{rmk}

\section{Representations and modular forms of rank two}
\label{s:rank2}

Below we will take tensor products and symmetric cubes of modular forms for two dimensional representations of $\Gamma$. Since one can give a complete description of all modular forms in rank two, this allows us to describe a finite number of one-parameter families of modular forms of rank four. In preparation for this, we begin by briefly recalling the description of representations and modules of modular forms of rank two, which is outlined in \cite{Mason2} and \cite{FrancMason1}. See also \cite{TubaWenzl}, which classifies irreducible representations of $\Gamma$ up to rank five. We too shall focus only on the irreducible representations, as our ultimate goal is the classification of certain irreducible representations of rank four. Note that $\Gamma$ has only a finite number of isomorphism classes of reducible representations of rank two, so that irreducibility is not a serious restriction in rank two.

Let $\rho \colon \Gamma \to \GL_2(\CC)$ be irreducible. From \cite{Mason2} or \cite{TubaWenzl}, by changing bases we may assume that
\begin{align*}
  \rho(T) &= \twomat xx0y, & \rho(S) &= \zeta^{2a}\twomat{0}{-x}{y}{0}, & \rho(R) &= \xi^{-a}\twomat{0}{-1}{1}{1},
\end{align*}
where $\xi = e^{2\pi i/6}$, $\zeta = \xi^2$, $xy = \xi^a$, and $x^2-xy+y^2 \neq 0$. Swapping the parameters $x$ and $y$ yields an isomorphic representation, but otherwise these representations are pairwise nonisomorphic. Notice that $\det \rho(T) = \chi(T)^{2a}$ where $\chi$ is the character of $\eta^2$, so that $\det \rho = \chi^{2a}$.

In order to describe the corresponding modular forms, there are two cases: the case where $x = y$, and the case where $x \neq y$. We handle the case $x = y$ first. Let $\alpha \colon \Gamma \to \GL_2(\CC)$ denote the inclusion representation. Then one verifies that in this case $\rho \cong \alpha \otimes \chi^a$ or $\rho \cong \alpha \otimes \chi^{a+6}$, and the corresponding modular forms are described as in Example \ref{ex:missinglattice} above, although one must rescale by $\eta^{2a}$ or $\eta^{2a+12}$.

It remains to treat the case when $x \neq y$. This was handled in \cite{FrancMason1}, although general exponents were not treated there. Thus, we will recall the results of \cite{FrancMason1} and explain how they generalize to arbitrary exponents.

Let $L$ denote an exponent matrix for $\rho(T)$. Then Example 1 of \cite{FrancMason4} explains that the minimal weight $k_1$ where $M_{k_1}(\rho,L)$ is nonzero equals $k_1 = 6\Tr(L)-1$. If $F \in M_{k_1}(\rho,L)$ is nonzero, then we claim that $F, DF$ must be a free basis for $M(\rho,L)$ over $M(1)$. If not, since there are no nonzero modular forms in $M(1) = \CC[E_4,E_6]$ of weight $2$, we deduce by Theorem \ref{t:fmt} that $DF = 0$. But this means that $F$ is of the form $\eta^{2k_1}v$ for some vector $v \in \CC^2$. The transformation laws for $F$ and $\eta$ imply that $v$ spans a subrepresentation of $\rho$ isomorphic to $\chi^{2k_1}$, contradicting the irreducibility of $\rho$. Note that this kind of argument holds quite generally and is a classical part of the theory of monodromy of differential equations --- see \cite{Mason1} for details.

Now, as in \cite{FrancMason1}, $F$, $DF$ and $D^2F$ are linearly dependent over $M(1)$, by Theorem \ref{t:fmt}. Therefore the argument from \cite{FrancMason1} for canonical exponents applies to arbitrary exponents. Recall that the idea is that we can write
\begin{equation}
\label{eq:mlderank2}
  D^2F + aE_4F = 0
\end{equation}
for some nonzero complex scalar $a$. Since $D\eta = 0$ (see e.g. the appendix to \cite{Katz2}), if we write $f = \eta^{-2k_1}F$, then $f$ is a possibly weakly holomorphic modular form of weight $0$ satisfying $D^2f + aE_4f = 0$. The form $f$ is a global section of the regular connection $\cV(\rho \otimes \chi^{-2k_1},L-\frac{k_1}{6})$, whose residue is $-L + \frac{k_1}{6}I_2$. If one writes $f$ as a multivalued function of $K = 1728/j$, where $j$ is the usual $j$ function, then \cite{FrancMason1} shows that $f(j)$ is the solution of an ordinary differential equation that is in fact hypergeometric. The indicial polynomial of this equation at $j=0$ is equal to the characteristic polynomial of $-L+\frac{k_1}{6}I_2$ (see Remark 3.12 of \cite{Katz}). One can use this to solve for the parameter $a$ in Equation \eqref{eq:mlderank2} in terms of the eigenvalues of $L$. By solving the hypergeometric differential equation satisfied by $f$, one can use the Free-Module Theorem \ref{t:fmt} to describe \emph{all} modular forms in $M(\rho,L)$. The precise result is the following:
\begin{thm}
\label{t:rank2}
Suppose that $\rho$ is an irreducible representation of $\Gamma$ of rank $2$, and let $L$ denote a choice of exponents for $\rho$, so that $\rho(T) = e^{2\pi i L}$. Then the following hold:
\begin{enumerate}
\item the minimal weight $k_1$ such that $M_{k_1}(\rho,L) \neq 0$ is $k_1 = 6\Tr(L)-1$;
\item if $F \in M_{k_1}(\rho,L)$ is nonzero, then every form in $M_k(\rho,L)$ can be described uniquely in the form $aF+bDF$ where $a \in M_{k-k_1}(1)$ and $b \in M_{k-k_1-2}(1)$;
\item if $\rho \cong \alpha \otimes \chi^a$ where $\alpha \colon \Gamma \to \GL_2(\CC)$ is the inclusion representation, then a minimal weight form for $\rho$ can be described as $PF$ for some $P \in \GL_2(\CC)$, where
  \[
  F(\tau) = \eta(\tau)^{2k_1+2}\twovec {\tau}{1};
\]
\item if $\rho(T)$ has distinct eigenvalues, so that $L$ also has distinct eigenvalues $r_1$ and $r_2$, then a minimal weight form for $\rho$ can be described as $PF$ for some $P \in \GL_2(\CC)$, where
  \[
  F = \eta^{2k_1}\twovec{K^{\frac{6(r_1-r_2)+1}{12}} {}_2F_1\left(\frac{6(r_1-r_2)+1}{12},\frac{6(r_1-r_2)+5}{12};r_1-r_2+1;K\right)}{K^{\frac{6(r_2-r_1)+1}{12}} {}_2F_1\left(\frac{6(r_2-r_1)+1}{12},\frac{6(r_2-r_1)+5}{12};r_2-r_1+1;K\right)}
\]
and $K = 1728/j$.
\end{enumerate}
\end{thm}
\begin{proof}
Parts (1) and (2) where proved in \cite{Mason2} for canonical exponents, and (4) was proved in \cite{FrancMason1} for canonical exponents. We have explained above how the computations from \cite{Mason2} and \cite{FrancMason1} generalize to handle arbitrary exponents. Part (3) is well-known for canonical exponents, and it was discussed in Example \ref{ex:missinglattice} above in general for $\rho = \alpha$. The other cases can be deduced from this by tensoring with powers of $\chi$, multiplying by powers of $\eta$, and then shifting exponents appropriately.
\end{proof}

\begin{rmk}
In part (3) of Theorem \ref{t:rank2}, if $\rho = \alpha \otimes \chi^a$ then the matrix $P$ is the identity. In part (4) the matrix $P$ depends on the isomophism class of $\rho$ and on the particular basis of solutions to the relevant hypergeometric differential equation that we used to describe $F$. If one is happy to work with any representation in the isomorphism class of $\rho$, then it is harmless to assume that $P$ is the identity in part (4) as well.
\end{rmk}

\begin{rmk}
Instances of the formula in part (4) of Theorem \ref{t:rank2} have been observed sporadically many times in mathematics. Some of those occurences are in fact quite classical. For example, the paper \cite{FrancMason1} was inspired by \cite{KanekoZagier}.
\end{rmk}

\section{Representations and modular forms of rank four}
\label{s:rank4}

Let $\rho \colon \Gamma \to \GL_4(\CC)$ denote a representation. Proposition 2.6 of \cite{TubaWenzl} yields a basis of the underlying vector space such that in that basis,
\begin{align*}
\rho(T) &= \left(\begin{matrix}
x&(1+D^{-1}+D^{-2})y&(1+D^{-1}+D^{-2})z&w\\
0&y&(1+D^{-1})z&w\\
0&0&z&w\\
0&0&0&w
\end{matrix}
\right),\\
\rho(B) &= \left(\begin{matrix}
w&0&0&0\\
-z&z&0&0\\
Dy&-(D+1)y&y&0\\
-D^3x&(D^3+D^2+D)x&-(D^2+D+1)x&x
\end{matrix}
\right),
\end{align*}
for scalars $x, y, w, z$ and where $D = \sqrt{yz/xw}$. Note that $S = B^{-1}T^{-1}B^{-1}$, and observe that the eigenvalues of $\rho(T)$ determine $\rho$ up to the choice of sign for $D$.

The Corollary above 2.10 of \cite{TubaWenzl} shows that $\det\rho(T)^3 = 1$. Let $\zeta = e^{2\pi i/3}$ and write $\det \rho(T) = \zeta^a$. Then $yz = \zeta^ax^{-1}w^{-1}$ and we have $D = \xi^{a} (xw)^{-1}$ or $D = \xi^{a+3}(xw)^{-1}$ where $\xi = e^{2\pi i/6}$. Let us thus write $D = \xi^d(xw)^{-1}$ for some $d \in \{0,\ldots, 5\}$, so that $xyzw = \zeta^d$. Note that if $L$ is a choice of exponents for $\rho$, then $3\Tr(L)$ is an integer satisfying $3\Tr(L) \equiv d \pmod{3}$. We have the identities
\begin{align*}
\Tr(\rho(S)) &= 0,& \Tr(\rho(R)) &= -\xi^{-d}, & \Tr(\rho(R^2)) &= \zeta^{-d},
\end{align*}
and $\rho(S^2) = -(-1)^{d}$. The odd representations correspond to even $d$. It will thus be convenient to write $\rho(-I) = (-1)^e$, so that $e \not \equiv d \pmod{2}$.
\begin{rmk}
In \cite{TubaWenzl} it is shown that certain choices of eigenvalues $x$, $y$, $z$ and $w$ \emph{do not} lead to irreducible representations. Moreover, permutations of the eigenvalues yield isomorphic representations. See Section 2.10 of \cite{TubaWenzl} for a precise description of the moduli space of irreducible representations of $\Gamma$ of rank $4$. We do not require this precise description of the moduli space.
\end{rmk}

In the applications below we will want to allow arbitrary choices of exponents.\ For example, if $\rho$ is the symmetric cube of a $2$-dimensional representation, then it is most natural to use the corresponding symmetric cube lift of the $2$-dimensional choice of exponents.\ These exponents need not agree with the canonical choice of exponents for $\rho$.\ The paper \cite{FrancMason4} explains how to compute the weights of a basis of modular forms in $M(\rho,L)$ for irreducible $\rho$ satisfying $\dim \rho \leq 5$ and arbitrary choices of exponents. In particular, when $\dim \rho = 4$ as in this paper, the argument used in the proof of Proposition 1 in \cite{FrancMason4} shows that for arbitrary exponents $L$ for $\rho$,
\[
  \dim_{\CC} M_{k}(\rho,L) = \begin{cases}
    0 & k < 3\Tr(L)-3,\\
    \chi(\cV_{k,L}(\rho)), & k \geq 3\Tr(L)-3,
  \end{cases}
\]
where $\chi(\cV_{k}(\rho,L))$ denotes the Euler characteristic
\[
\chi(\cV_{k}(\rho,L)) = \dim_{\CC} H^0(X,\cV_{k}(\rho,L)) - \dim_{\CC} H^1(X,\cV_{k}(\rho,L))
\]
of the bundle $\cV_{k}(\rho,L)$. Happily, the Euler characteristic is easy to compute. Corollary 6.2 of \cite{CandeloriFranc1} uses Riemann-Roch to obtain the following formula:
\begin{equation}
  \label{eq:eulerchar}
  \chi(\cV_{k}(\rho,L)) = \begin{cases}
   \frac{5+k-3\Tr(L)}{3}-\frac{\xi^{k-d}}{3(1-\zeta)}+\frac{\zeta^{k-d}}{3(1-\zeta^{-1})} & d\not \equiv k \pmod{2},\\
    0& d\equiv k \pmod{2}.
  \end{cases}
\end{equation}

\begin{thm}
  \label{t:weights}
  Let $\rho$ denote an irreducible representation of $\Gamma$ of rank $4$, and write $\rho(-I) = (-1)^e$. Let $L$ denote a choice of exponents for $\rho$, and let $k_1 \leq k_2 \leq k_3 \leq k_4$ denote the weights of a free basis for $M(\rho,L)$ over the ring $M(1)$ of modular forms of level one. Then one of the following is true:
\begin{enumerate}
\item \emph{Cyclic case}: if $3\Tr(L)\not\equiv e \pmod{2}$, then $k_1 {=} 3\Tr(L) {-} 3$ and
  \[
  (k_1,k_2,k_3,k_4) = (k_1,k_1+2,k_1+4,k_1+6);
\]
\item \emph{Noncyclic case}: if $3\Tr(L)\equiv e \pmod{2}$, then $k_1 {=} 3\Tr(L){-}2$ and
  \[
  (k_1,k_2,k_3,k_4) = (k_1,k_1{+}2,k_1{+}2,k_1{+}4).
  \]
\end{enumerate}
\end{thm}
\begin{proof}
Recall that $3\Tr(L) \equiv d \pmod{3}$. Therefore, since $d \not \equiv e \pmod{2}$, we have $3\Tr(L) \equiv d \pmod{6}$ in the cyclic case, and $3\Tr(L) \equiv d+3\pmod{6}$ in the noncyclic case.
  
  Consider the Euler-Poincare series $\sum_{k \in \ZZ} \dim M_{k}(\rho,L)T^k$.  On one hand, 
\[
\sum_{k \in\ZZ} \dim M_{k}(\rho,L)T^k = \frac{T^{k_1}+T^{k_2}+T^{k_3}+T^{k_4}}{(1-T^4)(1-T^6)}.
\]
On the other hand, by our identification of $\dim M_{k}(\rho,L)$ with an Euler characteristic, we deduce that
\[
  \frac{\sum_j T^{k_j}}{(1-T^4)(1-T^6)} = \sum_{\substack{k\geq 3\Tr(L)-3\\ d \not \equiv k \pmod{2}}}\left(\frac{5+k-3\Tr(L)}{3}-\frac{\xi^{k-d}}{3(1-\zeta)}+\frac{\zeta^{k-d}}{3(1-\zeta^{-1})}\right)T^k.
\]

First suppose that $3\Tr(L) \equiv d\pmod{6}$, so that $3\Tr(L)-3 \not \equiv d \pmod{2}$. Let $k_0 = 3\Tr(L)-3$, so that we deduce
\begin{align*}
  \frac{\sum_j T^{k_j}}{(1-T^4)(1-T^6)} &= \sum_{u \geq 0}\left(\frac{5+k_0-3\Tr(L)+2u}{3}-\frac{\xi^{k_0+2u-d}}{3(1-\zeta)}+\frac{\zeta^{k_0+2u-d}}{3(1-\zeta^{-1})}\right)T^{k_0+2u}\\
  &=T^{k_0}\sum_{u \geq 0}\left(\frac{2+2u}{3}-\frac{\xi^{k_0+2u-d}}{3(1-\zeta)}+\frac{\zeta^{k_0+2u-d}}{3(1-\zeta^{-1})}\right)T^{2u}\\
&=T^{k_0}\left(\frac{2}{3(1-T^2)^2}-\frac{1}{3(1-\zeta)}\frac{1}{1-\zeta T^2}-\frac{\zeta}{3(1-\zeta)}\frac{1}{1-\zeta^2T^2}\right)\\
&= \frac{T^{k_0}+T^{k_0+2}+T^{k_0+4}+T^{k_0+6}}{(1-T^4)(1-T^6)}
\end{align*}
It follows that the weights of the generators are as claimed. The proof in the noncyclic case when $3\Tr(L) \equiv d+3 \pmod{6}$ is analogous.
\end{proof}

\begin{ex}
  For fixed $\rho$ it is possible for both cases of Theorem \ref{t:weights} to occur as the exponents vary. To be concrete, \cite{TubaWenzl} shows that there exist real numbers $0 \leq a < b < c < d < 1$ such there is an irreducible representation $\rho$ of $\Gamma$ of rank $4$ with $\rho(T) = \diag(e^{2\pi i a}, e^{2\pi ib}, e^{2\pi i c}, e^{2\pi i d})$. Consider the two choices of exponents
  \begin{align*}
    L_1 &= \diag (a,b,c,d),\\
    L_2 &= \diag (a+1,b,c,d).
  \end{align*}
  Then $3\Tr(L_1) \not \equiv 3\Tr(L_2) \pmod{2}$, so that both cases of Theorem \ref{t:weights} are realized by $M(\rho,L_1)$ and $M(\rho,L_2)$.
\end{ex}
\begin{ex}
  In the opposite direction to the previous example, for some exceptional representations $\rho$ all choices of exponents give rise to the same case of Theorem \ref{t:weights}. For example, there is a choice of $\rho$ in the isomorphism class of the symmetric cube of the inclusion $\Gamma \hookrightarrow \GL_2(\CC)$ such that $\rho(T)$ is a single Jordan block with eigenvalue $1$. The possible choices of exponent matrices are then of the form $L_n = n + N$ where $n \in \ZZ$ and $N$ is nilpotent. Hence $3\Tr(L_n) = 12n \equiv 0 \pmod{2}$. However, $\rho$ is an odd representation, so that $(\rho,L_n)$ always corresponds to the cyclic case of Theorem \ref{t:weights}. 
\end{ex}

\section{The cyclic case}
\label{s:cyclic}
Let $\rho$ be an irreducible representation of $\Gamma$ of rank four, let $L$ denote a choice of exponents for $\rho$, and assume that $M(\rho,L)$ is cyclic as in part (1) of Theorem \ref{t:weights}. By Theorem \ref{t:weights}, the least weight of a  nonzero holomorphic modular forms for $\rho$ is $k_1= 3\Tr(L){-}3$, and such a form is \emph{unique} up to rescaling. Let $F$ be a nonzero form in the $1$-dimensional vector space $M_{k_1}(\rho,L)$. By cyclicity, $F$ must satisfy a differential equation of the form
\[
  D^4F + aE_4D^2F + bE_6DF + cE_4^2F = 0,
\]
for scalars $a,b,c \in \CC$ (see Lemma 8 of \cite{FrancMason4}). Following the computations leading to Example 16 of \cite{FrancMason2}, if $K= 1728/j$ and $\theta = Kd/dK$, then the form $\tilde F = \eta^{{-}2k_1}F$ of weight zero satisfies the ordinary differential equation
\begin{align}
  \label{eq:cyclicODE} 	
 \theta^{4}\tilde F - \left(\frac{2 K + 1}{1-K}\right) \theta^{3}\tilde F +
\left(\frac{44 K^{2} - \left( 36a + 28\right) K + 36a + 11}{36(1-K)^2}\right) \theta^{2}\tilde F& \\
\nonumber  + \left(\frac{8K^{2} - \left(12a +  36b + 4\right) K - 6 a+ 36b - 1}{36(1-K)^2}\right) \theta \tilde F + \frac{c}{(1-K)^2}\tilde F&=0.
\end{align}

We would like to find expressions for the constants $a$, $b$ and $c$ in equation \eqref{eq:cyclicODE} in terms of the data of the monodromy representation $\rho$ and the exponents $L$. The key point is that the differential equation \eqref{eq:cyclicODE} has regular singularities at $0$, $1$ and $\infty$. What is classically known as the indicial equation of \eqref{eq:cyclicODE} at a singular point can be computed as the characteristic polynomial of the matrix of exponents $L$ used to extend the flat bundle $\cV(\rho)$ to the cusp. For a clear explanation of the relationship between the classical indicial equation and the exponent matrix $L$, or residue of the holomorphic connection $\cV(\rho,L)$, see Section VI of \cite{Katz}.

One slightly technical point is that  \eqref{eq:cyclicODE} is the differential equation satisfied by the rescaled form $\tilde F$ of weight zero. Multiplying $F$ by $\eta^{-2k_1}$ to obtain $\tilde F$ corresponds to tensoring the bundle $\cV_{k_1}(\rho,L)$ with a line bundle, which at the level of exponents amounts to nothing more than shifting the exponents at the cusp by $-\frac{1}{12}k_1 = \frac{1}{4}(1-\Tr(L))$. See Remark 3.12 of \cite{CandeloriFranc1} for a discussion of this point. Thus, if the exponents of $\rho(T)$ are $e_1$, $e_2$, $e_3$ and $e_4$, so that $\Tr(L) = \sum e_j$, then the exponents of equation \eqref{eq:cyclicODE} are $f_j = e_j + \frac{1}{4}(1-\Tr(L))$ for $j=1,2,3,4$. On the other hand, since $K = 1728/j$, on the $K$-line the cusp of $X$ corresponds to $K= 0$. Since $\theta = Kd/dK$, the indicial polynomial of equation \eqref{eq:cyclicODE} at the cusp is
\[
  \prod_{j=1}^4\left(x-f_j\right) = x^4-x^3+\left(a+\tfrac{11}{36}\right)x^2 -\left(\tfrac 16 a -b+\tfrac 1{36}\right)x + c. 
\]
Therefore, if the eigenvalues of $L$ are $e_1$, $e_2$, $e_3$, $e_4$, then we find that
\begin{align*}
  a &= \sigma_2(f){-}\tfrac {11}{36},\\
  b &= -\sigma_3(f){+}\tfrac{1}{6}a {+} \tfrac{1}{36},\\
  c&= \sigma_4(f),
\end{align*}
where $f_j = e_j + \frac{1}{4}(1-\Tr(L))$ and $\sigma_d(f)$ is the usual elementary symmetric polynomial of degree $d$ in the variables $f_1$, $f_2$, $f_3$, $f_4$. Thus, from $\rho$ and $L$ we can write down a precise differential equation such that a basis of solutions  to the equation are the coordinates of a minimal weight form in $M_{k_1}(\rho,L)$. Then by cyclicity, an $M(1)$-basis of forms for $M(\rho,L)$ will be given by $F$, $DF$, $D^2F$ and $D^3F$, where $D$ denotes the modular derivative.

Unfortunately it is not easy to solve  \eqref{eq:cyclicODE} in general. For example, no specialization of $a$, $b$ and $c$ leads to a generalized hypergeometric equation in the sense of \cite{BeukersHeckman}. In Sections \ref{s:tensorproduct}, \ref{s:symmetriccube} and \ref{s:induction} below we will use functorial linear algebraic constructions to solve this equation in many cases, and thereby describe the corresponding spaces of vector-valued modular forms.

\section{The noncyclic case}
\label{s:noncyclic}

In this section $\rho$ still denotes an irreducible representation of $\Gamma$ of rank four, and $L$ denotes a choice of exponents for $\rho$, but now we consider the case when $M(\rho,L)$ is \emph{not} a cyclic module over $M(1)\langle D\rangle$. In this case Theorem \ref{t:weights} shows that there exist generators for $M(\rho,L)$ in weights $k_1$, $k_1+2$, $k_1+2$ and $k_1+4$, where $k_1 = 3\Tr(L)-2$.

\begin{lem}
Suppose that $(\rho,L)$ is an irredicible presentation $\rho$ of rank $4$ and a choice of exponents $L$ for $\rho$, such that the noncyclic case of Theorem \ref{t:weights} holds. Then there exists an $M(1)$-basis for $M(\rho,L)$ of the form $F$, $DF$, $G$, $H$, with $F \in M_{k_1}(\rho,L)$, $G \in M_{k_1+2}(\rho,L)$ and $H \in M_{k_1+4}(\rho,L)$, such that the matrix of the modular derivative $D$ in this basis satisfies
\begin{equation}
\label{eq:matrixMLDE}
  D(F,DF,G,H) = (F,DF,G,H)\left(\begin{matrix}
      0&aE_4&E_4&0\\
      1&0&0&bE_4\\
      0&0&0&cE_4\\
      0&1&0&0
    \end{matrix}
  \right),
\end{equation}
for complex scalars $a$, $b$ and $c$, with $c \neq 0$.
\end{lem}
\begin{proof}
As in Section \ref{s:cyclic}, let $F$ denote a nonzero form in the $1$-dimensional vector space $M_{k_1}(\rho,L)$. Observe that as above, since $\rho$ is irreducible, $DF$ is nonzero. It follows that there is a form $G$ of weight $k_1+2$ such that $DF$ and $G$ constitute a $\CC$-basis for $M_{k_1+2}(\rho,L)$. Let $H\in M_{k_1+4}(\rho,L)$ complete $F,DF,G$ to a free basis for $M(\rho,L)$. Observe that $D^2F=aE_4F+bH$ and $DG=cE_4F+dH$ for scalars $a$, $b$, $c$, $d$, and $b$ must be nonzero by irreducibility of $\rho$. Thus, we may as well take $b=1$ so that $D^2F = aE_4F + H$. After this adjustment, we can add a multiple of $DF$ to $G$ in order to assume that $DG = cE_4F$, where again $c \neq 0$. If necessary, we may rescale $G$ to achieve $DG = E_4F$. Finally, after replacing $H$ by some multiple $H-\alpha E_4F$, we may assume that $DH$ is a linear combination of $E_4DF$ and $E_4G$. This establishes that we may find a free basis for $M(\rho,L)$ of type $(F, DF, G, H)$ as in the Lemma (after a relabelling of variables).
\end{proof}

\begin{rmk}
  Observe that the existence of three free parameters above matches up with the fact that the moduli space of $4$-dimensional irreducible representations of $\Gamma$ is $3$-dimensional.
\end{rmk}

As in the cyclic case, equation \eqref{eq:matrixMLDE} suffers from the fact that it does not involve modular forms of weight $0$, and hence it does not concern sections of a flat bundle.\ Rather, it concerns sections of twists of a flat bundle.\ To get around this we will instead look for the differential equation satisfied by $F$, $(E_4/E_6)DF$, $(E_4/E_6)G$ and $(1/E_4)H$.\ A straightforward computations then shows that 
\[\cF = (F,A^{-1}DF, A^{-1}G,E_4^{-1}H)\]
satisfies the matrix differential equation
\[
  D\cF = \cF\left(\begin{matrix}
      0&aE_4A^{-1}&E_4A^{-1}&0\\
      A&-D(A)A^{-1}&0&bA\\
      0&0&-D(A)A^{-1}&cA\\
      0&E_4A^{-1}&0&\frac{1}{3}A
    \end{matrix}
  \right)
\]
where $A = E_6/E_4$. Now all entries in the matrix above are of weight $2$, at the expense of our having introduced some poles at singular points corresponding to the zeros of $E_4$ and $E_6$. In this equation both $\cF$ and the differential operator $D$ are in weight $k_1$.  Again using that $D(\eta) = 0$, we can replace $\cF$ by $\tilde \cF = \eta^{-2k_1}\cF$ to shift this equation to weight $0$. After this change, $D=D_0 = A\theta_K$ for $\theta_K = Kd/dK$ and we have
\begin{align*}
 E_4 &=\frac{A^2}{1-K}, & E_6 &=\frac{A^3}{1-K}, & D(A) &= -A^2\frac{1+2K}{6(1-K)}.
\end{align*}
Thus, after performing these substitutions, solving equation \eqref{eq:matrixMLDE} is equivalent to solving the following matrix ordinary differential equation on the $K$-line:
\begin{equation}
  \label{eq:noncyclicMODE}
  \theta_K\tilde \cF = \tilde \cF\left(\begin{matrix}
      0&\frac{a}{1-K}&\frac{1}{1-K}&0\\
      1&\frac{1+2K}{6(1-K)}&0&b\\
      0&0&\frac{1+2K}{6(1-K)}&c\\
      0&\frac{1}{1-K}&0&\frac{1}{3}
    \end{matrix}
  \right).
\end{equation}

As before we would like to get expressions for $a$, $b$ and $c$ in terms of the data of the representation $\rho$ and the exponent matrix $L$. Again, we use the indicial equation at the cusp, corresponding to $K=0$. If $e_1$, $e_2$, $e_3$ and $e_4$ are the eigenvalues of the exponent matrix $L$, then define
\[f_j = e_j -\frac{1}{12}k_1 = e_j -\frac{1}{4}\Tr(L) + \frac{1}{6}.\]
The indicial equation at $K=0$ of the matrix differential equation \eqref{eq:noncyclicMODE} is the characteristic polynomial of the matrix obtained by setting $K = 0$ in \eqref{eq:noncyclicMODE}. By comparison with the description \cite{Katz} of the indicial polynomial as the characteristic polynomial of the residue $L$ of the flat connection corresponding to $\rho$ (or rather, corresponding to the twist of $\rho$ by the character of $\eta^{-2k_1}$), we obtain the identity
\begin{align*}
\prod_{j=1}^4(x-f_j) &=  \textrm{charpoly}\left(\begin{matrix}
      0&a&1&0\\
      1&\frac{1}{6}&0&b\\
      0&0&\frac{1}{6}&c\\
      0&1&0&\frac{1}{3}
    \end{matrix}
  \right)\\
	&= x^{4} - \frac{2}{3} x^{3} + \left(- a -  b + \frac{5}{36}\right) x^{2} + \left(\frac{1}{2} a + \frac{1}{6} b - \frac{1}{108}\right) x - \frac{1}{18} a -  c
\end{align*}
Therefore, we deduce that
\begin{align*}
  a &= -3\sigma_3(f)+\frac{1}{2}\sigma_2(f) -\frac{1}{24},\\
  b&= 3\sigma_3(f)-\frac{3}{2}\sigma_2(f)+\frac{13}{72},\\
  c &= -\sigma_4(f)-\frac{1}{18}a.
\end{align*}
This yields a precise matrix differential equation expressed in terms of the data of $\rho$ and $L$ whose solutions can be used to produce a basis for $M(\rho,L)$ over $M(1)$ as discussed above.

Next, by performing a cyclic vector computation using \emph{Sage}, we find that $\tilde F$ satisfies the following scalar differential equation:
\begin{align}
  \label{eq:noncyclicODE}
  \theta_K^4\tilde F  -\frac{(7K+2)}{3(1-K)}\theta^3_K\tilde F + \frac{(56K^2+36K(a+b)-34K-36(a+b)+5)}{36(1-K)^2}\theta_K^2\tilde F&\\
\nonumber  +\frac{(32K^2+36bK-22K+54a+18b-1)}{108(1-K)^2}\theta_K\tilde F-\frac{(2aK+a+18c)}{18(1-K)^2}\tilde F&=0
\end{align}

Unfortunately, like its cyclic predecessor \eqref{eq:cyclicODE}, equation \eqref{eq:noncyclicODE} cannot be solved in exact terms as a series (in general). In the remaining sections we will explain how to use tensor products, symmetric powers and induction to solve some of these equations, and thereby obtain explicit formulas for the corresponding modular forms.\ Unfortunately we can only treat certain one and two parameter families of representations this way.\ Since the moduli space of $4$-dimensional irreducible representations of $\Gamma$ is $3$-dimensional, this means that we miss most representations.\ In the remaining cases one can still solve equations \eqref{eq:cyclicODE} and \eqref{eq:noncyclicODE} recursively to obtain $q$-expansions of the corresponding modular forms, and this is often good enough for applications.\ We end this section by summarizing the steps for performing such computations.

Suppose given a $4$-dimensional complex representation $\rho$ of $\Gamma$ and an exponent matrix $L$ such that $\rho(T) = e^{2\pi iL}$. The module $M(\rho,L)$ of modular forms can be described as follows:
\begin{enumerate}
\item[(a)] Let $0 \leq d \leq 5$ satisfy $\Tr(\rho(R)) = -e^{-2\pi id/6}$.
\item[(b)] If $3\Tr(L) \equiv d\pmod{2}$ then $M(\rho,L)$ is cyclic with minimal weight $k_1 = 3\Tr(L)-3$. Otherwise $M(\rho,L)$ is noncyclic and $k_1 = 3\Tr(L)-2$.
\item[(c)] Compute the parameters $a,b,c$ of the relevant differential equation using the formulas from Section \ref{s:cyclic} or Section \ref{s:noncyclic}, depending on whether $M(\rho,L)$ is cyclic or not.
\item[(d)] Find a basis of solutions to the cyclic equation \eqref{eq:cyclicODE} or the noncyclic equation \eqref{eq:noncyclicODE} near $K=0$. In general one can only hope to recursively compute a finite number of Taylor coefficients.
\item[(e)] Let $\tilde F$ be the vector-valued function whose coordinates are the basis found in step (4). Substitute the $q$-expansion of $K = 1728/j$ into $\tilde F$ to obtain the $q$-expansion of $\tilde F$.
\item[(f)] Set $F = \eta^{2k_1}\tilde F$, which is then a modular form of minimal weight $k_1$.
\end{enumerate}
Due to the nonuniqueness of the choice of basis from step (d), there exists a conjugate pair $(\rho',L')$ to $(\rho,L)$ such that $F \in M_{k_1}(\rho',L')$. In the cyclic case, an $M(1)$-basis for $M(\rho',L')$ is given by $F$, $DF$, $D^2F$ and $D^3F$, where $D$ denotes the usual modular derivative. In the noncyclic case an $M(1)$-basis is given by $F$, $DF$, $H = D^2F-aE_4F$ and $G = \frac{1}{cE_4}(DH-bE_4DF)$. A basis for $M(\rho,L)$ can then be found by multiplying forms in $M(\rho',L')$ by an appropriately chosen change of basis matrix. It can be difficult to find such a change of basis matrix exactly, but thankfully in many applications (e.g. studying modular forms on finite index subgroups, or studying questions about unbounded denominators) one can work with $M(\rho',L')$ in place of $M(\rho,L)$, without knowing \emph{exactly} what conjugate representation $\rho'$ is.

\section{Tensor products}
\label{s:tensorproduct}

Let $(\alpha,\beta)$ denote a pair of $2$-dimensional representations of $\Gamma$. If $f = (f_1,f_2)^T$ and $g = (g_1,g_2)^T$ are vector-valued forms of weight $k$ and $l$, respectively, that transform under $\alpha$ and $\beta$, respectively, then the form $F = (f_1g_1,f_1g_2,f_2g_1,f_2g_2)^T$ transforms under $\alpha \otimes \beta$ and is of weight $k+l$. In this section we show that when $\alpha \otimes \beta$ is irreducible and  $f$ and $g$ are of minimal weight, and if one works with the appropriate tensor product exponents of $\alpha \otimes \beta$, then $F$ is a form of minimal weight for $\alpha \otimes \beta$. Using Theorem \ref{t:rank2} and the results of Sections \ref{s:cyclic} and \ref{s:noncyclic}, this allows us to describe the corresponding module of modular forms of rank $4$.

First we want to identify when the tensor product $\alpha\otimes\beta$ of a pair of $2$-dimensional representations of $\Gamma$ is irreducible. In order to state the result we introduce a piece of notation and a piece of terminology. Let $\nu$ denote the standard inclusion representation $\nu\colon \Gamma \hookrightarrow \GL_2(\CC)$. If $\chi$ is a $1$-dimensional representation of $\Gamma$ then write $\nu_\chi = \nu \otimes \chi$. 
\begin{dfn}
  \label{d:Tregular}
A representation $\rho$ of $\Gamma$ is said to be \emph{$T$-regular} provided that the eigenvalues of $\rho(T)$ are \emph{pairwise distinct}.
\end{dfn}

Recall from Section \ref{s:rank2} that every $2$-dimensional irreducible representation $\rho$ of $\Gamma$ is either $T$-regular, or else $\rho \cong \nu_\chi$ for some  $\chi$. 
\begin{thm}
  \label{t:irrtp}
  Let $\alpha$ and $\beta$ denote a pair of $2$-dimensional representations of $\Gamma$. Then $\alpha\otimes \beta$ is irreducible precisely when $\alpha$ and $\beta$ are both irreducible, and one of the following additional conditions holds:
  \begin{enumerate}
  \item exactly one of $\alpha$ or $\beta$ is $T$-regular;
  \item all three of $\alpha$, $\beta$ and $\alpha \otimes \beta$ are $T$-regular.
  \end{enumerate}
\end{thm}
\begin{proof}
Clearly, $\alpha$ and $\beta$ must be irreducible. We may assume that each of the $T$-matrices $\alpha(T)$ and $\beta(T)$ is either diagonal with distinct eigenvalues, or else it is a single Jordan block of the form $\stwomat{\lambda}{1}{0}{\lambda}$.\ The latter type only appear when the $2$-dimensional irreducible in question is isomorphic with some $\nu_\chi$. Therefore, if \emph{both} $\alpha(T)$ and $\beta(T)$ are single Jordan blocks, then $\rho = \alpha\otimes\beta$ is equivalent to $\nu\otimes \nu \otimes \psi$ for some $1$-dimensional representation $\psi$. This is \emph{not} irreducible, since $\nu^{\otimes 2}\cong S^2\nu\oplus \Lambda^2\nu$.

Assume that (1) holds, so that just one of $\alpha(T)$ and $\beta(T)$, let's say the latter, is a single Jordan block. To prove that $\rho$ is irreducible, we may assume without loss of generality that $\beta = \nu$ and $\alpha(T)$ is diagonal with distinct eigenvalues.\ Then $\rho(T)$ has eigenvectors that span a $2$-dimensional subspace $E$ of the linear space $V$ furnished by $\rho$, and if $E'=\rho(S)E$ then $E\cap E'=0$. Now suppose that $W \subseteq V$ is a nonzero $\Gamma$-submodule. If $W$ contains two eigenvectors for $\rho(T)$ then $E \subseteq W$ implies that $E\oplus E' \subseteq W$, an impossibility. Otherwise we must have $\dim W = 2$ and $W$ contains a single eigenvector. A straightforward computation shows that this is also impossible.

Thus, it remains to consider the case when $\alpha$ and $\beta$ are both irreducible and $T$-regular. We must show in this case that $\rho$ is irreducible if, and only if, it is $T$-regular. Observe that $\rho(T)$ is diagonal and $\rho$ is 4-dimensional, so that the main Theorem of \cite{TubaWenzl} implies that if $\rho$ is irreducible, then $\rho$ is $T$-regular.

Conversely, suppose that all of $\alpha$, $\beta$ and $\rho$ are $T$-regular. Without loss of generality we may change bases and assume that:
\begin{align*}
\alpha(T)&=\left(\begin{array}{cc} e^{2\pi ir} & 0 \\0 & e^{2\pi i s}\end{array}\right), &
\alpha(S)&= \left(\begin{array}{cc}a & b \\ c & -a\end{array}\right)\\
\beta(T)&=\left(\begin{array}{cc} e^{2\pi it} & 0 \\0 & e^{2\pi i u}\end{array}\right), &
\beta(S)&= \left(\begin{array}{cc} e & f \\ g& -e\end{array}\right)\\
\rho(T) &= \left(\begin{array}{cccc}e^{2\pi ix_1} & 0 & 0 & 0 \\0 & e^{2\pi ix_2}  & 0 & 0 \\0 & 0 &
e^{2\pi ix_3}  & 0 \\0 & 0 & 0 & e^{2\pi i x_4} \end{array}\right), & \rho(S)&= \left(\begin{array}{cccc}ae & be & af & bf \\ce & -ae & cf & -af \\ ag & bg & -ae &-be  \\ cg & -ag & -ce & ae\end{array}\right).
\end{align*}
Then since $\alpha$ and $\beta$ are $2$-dimensional irreducible representations of $\Gamma$, we have $a^2+bc=\pm 1$,  $e^2+fg= \pm 1$ and $abcefg\neq 0$.

Assume by way of contradiction that $\rho$ is not irreducible. Any proper invariant subspace is a sum of eigenspaces for $T$. But clearly if $v$ is an eigenvector for $T$ then $\rho(S)v$ projects nontrivially onto every $T$-eigenspace on account of the non-vanishing of the entries for $\rho(S)$. This contradiction completes the proof of the Theorem.
\end{proof}

Now let $\alpha$ and $\beta$ be two-dimensional irreducible representations, and let $L_1$ and $L_2$ denote choices of exponents for each.\ On the open modular curve $Y = \Gamma \backslash \uhp$, the regular connections corresponding to $\alpha$, $\beta$ and $\alpha \otimes \beta$ satisfy
\[
  \cV(\alpha)\otimes \cV(\beta) \cong \cV(\alpha\otimes \beta).
\]
As discussed in Section \ref{s:vvmfs}, the choices of exponents define extensions of these vector bundles to the closed modular curve $X$. Then $\cV(\alpha,L_1) \otimes \cV(\beta,L_2)$ is \emph{an} extension of $\cV(\alpha\otimes \beta)$ to the cusp. Example \ref{ex:missinglattice} showed that this bundle need not be of the form $\cV(\alpha\otimes \beta,\Lambda)$ for any choice of exponents $\Lambda$ for $\alpha\otimes \beta$. However, a standard computation at the cusp, recalled below, shows that $\cV(\alpha,L_1)\otimes \cV(\beta,L_2)$ \emph{is} of this form for a precise choice of exponents:
\begin{dfn}
  \label{d:tensorproductexponents}
  The \emph{induced tensor product exponents} are the natural choice of exponents $L_1 \tpexp L_2$ for $(\alpha \otimes \beta)(T)$ such that
  \[
  \cV(\alpha,L_1)\otimes \cV(\beta,L_2) \cong \cV(\alpha \otimes \beta,L_1\tpexp L_2).
  \]
\end{dfn}
\begin{rmk}
  Beware that $L_1{\tpexp} L_2$ is \emph{not} the tensor product of the matrices $L_1$ and $L_2$. Also, in Definition \ref{d:tensorproductexponents} we do not need to assume that $\alpha$ and $\beta$ are rank $2$ or irreducible.
\end{rmk}

We next describe the tensor product exponents quite explicitly. This is a standard computation, but we include the details for completeness and ease of reading.

In full generality, assume that $\alpha$ and $\beta$ are of ranks $m$ and $n$, respectively.\ Recall that meromorphic modular functions $F$ and $G$ for $\alpha$ and $\beta$, respectively, extend holomorphically to the cusps in $\cV(\alpha,L_1)$ and $\cV(\beta,L_2)$ if, and only if, the functions $\tilde F(\tau) = e^{-2\pi iL_1\tau}F(\tau)$ and $\tilde G(\tau) = e^{-2\pi iL_2\tau}G(\tau)$ have holomorphic $q$-expansions. Observe that $L_1 \otimes I_n$ and $I_m \otimes L_2$ commute. Therefore,
\begin{align*}
  e^{-2\pi i (L_1\otimes I_n + I_m\otimes L_2)\tau}f(\tau)\otimes g(\tau) &= e^{-2\pi i I_m\otimes L_2\tau}e^{-2\pi iI_m\otimes L_2\tau}f(\tau)\otimes g(\tau)\\
                                                                          &=(I_m \otimes e^{-2\pi i L_2\tau})(e^{-2\pi i L_1\tau}\otimes I_n)f(\tau)\otimes g(\tau)\\
  &= (e^{-2\pi i L_1\tau}f(\tau))\otimes (e^{-2\pi i L_2\tau}g(\tau)).
\end{align*}
It follows that if $F$ and $G$ are local sections of $\cV(\alpha)$ and $\cV(\beta)$ near the cusp, so that $F \otimes G$ is a local section of $\cV(\alpha {\otimes} \beta)$ near the cusp, then $F\otimes  G$ extends to $\cV(\alpha \otimes \beta, L_1\otimes I_n +I_m\otimes L_2)$ if $f$ and $g$ extend to $\cV(\alpha,L_1)$ and $\cV(\beta,L_2)$ respectively. Therefore there is an inclusion
\[
  \cV(\alpha,L_1) \otimes \cV(\beta,L_2) \subseteq \cV(\alpha \otimes \beta, L_1{\otimes} I_n +I_m{\otimes} L_2)
\]
of sheaves. This inclusion is in fact an equality: suppose that a pure tensor $f\otimes g$ extends to $\cV(\alpha \otimes \beta,L_1\otimes I_n +I_m\otimes L_2)$, so that the product $(e^{-2\pi i L_1\tau}f(\tau)) \otimes (e^{-2\pi i L_2\tau}g(\tau))$ is holomorphic at the cusp. Assume at least one of the factors is not holomorphic at the cusp, say the first has a pole of order $N$. Then the second factor must vanish to order at least $N$ at the cusp. Observe that $f\otimes g = (f/j^N)\otimes (j^Ng)$ where $j$ denotes the usual $j$-function. Then $f/j^N$ extends to $\cV(\alpha,L_1)$ and $j^Ng$ extends to $\cV(\beta,L_2)$, so that we in fact have an equality of sheaves above (it suffices to treat pure tensors since a basis of flat sections at the cusp can be constructed using pure tensors). This identifies the tensor product exponents $L_1\otimes_e L_2$ as:
\begin{align}
  \label{eq:tensorproductexponents}
  L_1\tpexp L_2 &= L_1\otimes I_n {+} I_m\otimes L_2,\\
  \Tr(L_1\tpexp L_2) &= n\Tr(L_1){+}m\Tr(L_2).
\end{align}
Note that this argument works for general curves, with $j$ replaced by any nonzero holomorphic function with a pole of order $1$ at the cusp.

\begin{rmk}
  Observe that even if $L_1$ and $L_2$ are canonical, so that the real parts of the eigenvalues of both of them are contained in $[0,1)$, the induced tensor product exponents need not be canonical.
\end{rmk}

Now, for any irreducible representations $\alpha$ and $\beta$ of $\Gamma$ of rank $2$, Section \ref{s:rank2} provides a basis such that
\begin{align*}
  \alpha(T) &= \twomat xx0y, & \alpha(S) &= \zeta^{2a}\twomat{0}{-x}{y}{0}, & \alpha(R) &= \xi^{-a}\twomat{0}{-1}{1}{0},
\end{align*}
where $\xi = e^{2\pi i/6}$, $\zeta = \xi^2$, $xy = \xi^a$, and $x^2-xy+y^2 \neq 0$. Choose a similar basis for $\beta$ and write $\det \beta(T) = \xi^b$. If $\alpha(-1) = (-1)^{e_1}$ and $\beta(-1) = (-1)^{e_2}$ then $(\alpha \otimes \beta)(-1) = (-1)^{e_1+e_2}$. Note that $e_1 \equiv a+1\pmod{2}$ and $e_2 \equiv b+1\pmod{2}$, so that $e_1+e_2 \equiv a+b\pmod{2}$. We next show that the tensor product of minimal weight forms in $M(\alpha,L_1)$ and $M(\beta,L_2)$ defines a minimal weight form in $M(\alpha\otimes \beta,L_1\tpexp L_2)$ whenever $\alpha \otimes \beta$ is irreducible.
\begin{thm}
  \label{t:tpnoncyclic}
  Let $\alpha$ and $\beta$ denote irreducible representations of $\Gamma$ of rank $2$ such that $\alpha \otimes \beta$ is also irreducible. Let $L_1$ and $L_2$ denote choices of exponents for $\alpha$ and $\beta$, respectively. Then the following hold:
  \begin{enumerate}
    \item the $M(1)\langle D\rangle$-module $M(\alpha\otimes \beta, L_1\tpexp L_2)$ is \emph{not} cyclic;
    \item the minimal weight for $(\alpha\otimes \beta, L_1\tpexp L_2)$ is $k+l$ where $k = 6\Tr(L_1){-}1$ and $l = 6\Tr(L_2){-}1$ are the minimal weights for $(\alpha,L_1)$ and $(\beta, L_2)$, respectively;
    \item if $A$ and $B$ denote forms of minimal weight for $(\alpha,L_1)$ and $(\beta,L_2)$, respectively, then up to conjugation of $\alpha \otimes \beta$, an $M(1)$-basis for $M(\alpha \otimes \beta,L_1 \tpexp L_2)$ is given by the forms
      \begin{align*}
        F  &= A\otimes B,\\
        DF &= D(A)\otimes B+ A\otimes D(B),\\
        H &= D^2F-aE_4A\otimes B,\\
        G &= \frac{1}{cE_4}(DH-bE_4DF),        
      \end{align*}
      where the complex scalars $a$, $b$, and $c$ are computed from $(\alpha\otimes \beta,L_1 \tpexp L_2)$ as in Section \ref{s:noncyclic}.
    \end{enumerate}
  \end{thm}
\begin{proof}
  First we prove (1). To apply Theorem \ref{t:weights} we must compute $3\Tr(L_1\tpexp L_2) = 6(\Tr(L_1)+\Tr(L_2))$ and compare it with the sign of $\alpha \otimes \beta$, which we have seen is congruent to $a + b\pmod{2}$. First suppose that $\alpha$ is $T$-regular and $\beta = \nu_\chi$ as in Case (1) of Theorem \ref{t:irrtp}. Then without loss of generality we can write
  \begin{align*}
  L_1 &= \twomat {u+r}{0}{0}{\frac{a}{6}-u+r'}, & L_2 &= \twomat{\frac{b}{12}+s}{1}{0}{\frac{b}{12}+s},
  \end{align*}
  for $u \in \CC$ and $r$, $r'$, $s \in \ZZ$. Therefore,
  \[
3\Tr(L_1\otimes_e L_2) = a+b+6r+6r'+12s \equiv a+b \pmod{2}.
\]
Hence we are in the noncyclic case of Theorem \ref{t:weights}, as claimed.

If instead $\alpha$, $\beta$ and $\alpha\otimes \beta$ are all $T$-regular, then we have
\begin{align*}
  L_1 &= \twomat {u+r}{0}{0}{\frac{a}{6}-u+r'}, & L_2 &= \twomat {v+s}{0}{0}{\frac{b}{6}-v+s'},
\end{align*}
for $u$, $v \in \CC$ and $r$, $r'$, $s$, $s' \in \ZZ$. Thus $3\Tr(L_1\tpexp L_2) \equiv a+b \pmod{2}$ and so the noncyclic case of Theorem \ref{t:weights} holds again. By Theorem \ref{t:irrtp}, this shows that $M(\alpha\otimes \beta,L_1\tpexp L_2)$ is noncyclic in all cases where $\alpha\otimes \beta$ is irreducible.

Next we prove (2). Recall from Section \ref{s:vvmfs} that if $\alpha$ is any two-dimensional irreducible representation of $\Gamma$ and $L_1$ is any choice of exponents for $\alpha$, then the minimal weight for $(\alpha,L_1)$ is $6\Tr(L_1)-1$. Therefore, by the noncyclic case of Theorem \ref{t:weights}, the minimal weight for $M(\alpha \otimes \beta,L_1\tpexp L_2)$ is
\[
  3\Tr(L_1\tpexp L_2)-2 = 3(2\Tr(L_1)+2\Tr(L_2))-2 = k+l
\]
as claimed.

Finally, (2) shows that the tensor product form $F = A\otimes B$ is of minimal weight for $(\alpha\otimes \beta,L_1\tpexp L_2)$. Section \ref{s:noncyclic} explained how to use a minimal weight form to obtain an $M(1)$-basis in the noncyclic case, and this is where the formulae of (3) arise from.
\end{proof}
\begin{rmk}
  Using Theorem \ref{t:rank2}, we can write down explicit hypergeometric formulas for the forms $A$ and $B$ in Theorem \ref{t:tpnoncyclic}, and thereby obtain explicit formulas for all elements in $M(\alpha \otimes \beta,L_1\otimes_eL_2)$. Since every form in $M^\dagger(\alpha \otimes \beta)$ is contained in some lattice of the form $M(\alpha \otimes \beta,L_1\otimes_eL_2)$, one can in this way describe all weakly holomorphic modular forms for $\alpha \otimes \beta$. 
\end{rmk}

\section{Symmetric cubes}
\label{s:symmetriccube}
In this section we let $\alpha$ denote an irreducible representation of $\Gamma$ of dimension $2$, and we set $\rho = \Sym^3\alpha$, so that $\rho$ is a $4$-dimensional representation. Suppose that
\[
  A = \twovec fg
\]
is a vector valued modular form for $\alpha$ of weight $k$. Then
\[
  F = \Sym^3A = \left(\begin{matrix}
      f^3\\
      f^2g\\
      fg^2\\
      g^3
    \end{matrix}
  \right)
\]
is a vector-valued modular form for $\rho$ of weight $3k$. In order to connect the symmetric cube form to the theory above we must first determine when $\rho$ is irreducible. The story is similar to the case of tensor products.
\begin{thm}
  \label{t:irrsymcube}
  Let $\alpha$ denote a $2$-dimensional representation of $\Gamma$.\ Then $\Sym^3\alpha$ is irreducible if and only if $\alpha$ is irreducible, and one of the following additional conditions is satisfied:
  \begin{enumerate}
  \item $\alpha \cong \nu_\chi$ for some $1$-dimensional character $\chi$ of $\Gamma$;
    \item $\alpha$ and $\Sym^3\alpha$ are both $T$-regular.
  \end{enumerate}
\end{thm}
\begin{proof}
  The proof is similar to the proof of Theorem \ref{t:irrtp}.\ For Case (1) it suffices to consider $\alpha = \nu$, in which case the irreducibility of $\Sym^n\nu$ for all $n \geq 0$ is a classical fact.

  If $\alpha \not \cong \nu_\chi$ then $\alpha$ is $T$-regular, and so $(\Sym^3\alpha)(T)$ is diagonalizable.\ By \cite{TubaWenzl}, we find that $\Sym^3\alpha$ must thus be $T$-regular in order to be irreducible.\ Finally, a direct computation using a description
\begin{align*}
  \alpha(T) &= \twomat xx0y, & \alpha(S) &= \zeta^{2a}\twomat 0{-x}y0,&\alpha(R) &= \xi^{-a}\twomat 0{-1}11,
\end{align*}
where $\xi = e^{2\pi i/6}$, $\zeta= \xi^2$, $xy=  \xi^a$ and $x^2-xy+y^2 \neq 0$, allows one to show that $\Sym^3\alpha$ is irreducible when $\alpha$ and $\Sym^3\alpha$ are both $T$-regular.
\end{proof}

As in the case of the tensor product, bundles of the form $\Sym^3\cV(\alpha,L)$ need not, in principle, be of the form $\cV(\Sym^3\alpha,\Lambda)$ for a choice of exponents $\Lambda$. But in fact this again turns out to be the case, and as for tensor products this is a classical computation that holds for all symmetric powers, all $\alpha$, and all choices of exponents:
\begin{dfn}
  If $\alpha$ is a representation of $\Gamma$ and if $L$ denotes a choice of exponents for $\alpha$, then the \emph{$n$th symmetric power} exponents $S^nL$ are the natural exponents for $\Sym^n\alpha$ such that
  \[
  \Sym^n\cV(\alpha,L) \cong \cV(\Sym^n\alpha,S^nL).
  \]
\end{dfn}
It is more cumbersome to write down general explicit formulas for the symmetric power exponents than it is for tensor products. They can be computed by fixing bases so that one has an explicit symmetric power map
\[
  \Sym^n \colon \GL_d(\CC) \to \GL_{e}(\CC).
\]
where  $e = \binom{d+n-1}{n}$. Then if $L$ is a choice of exponents for $\rho(T)$, the function $\Sym^ne^{2\pi i L\tau}$ gives a one-parameter subgroup through $\Sym^n\rho(T)$ at $\tau = 1$. We have
\[
  S^nL = \frac{1}{2\pi i}\left(\frac{d}{d\tau}\Sym^ne^{2\pi i L\tau}\right)\bigg |_{\tau =0}.
\]
For example, if $L = \stwomat {e_1}{e_2}{e_3}{e_4}$ denotes a choice of exponents for a two-dimensional $\alpha$, then one computes that
\[
  S^3L = \left(\begin{matrix}
      3e_1 & e_2 &0&0\\
      3e_3 & 2e_1+e_4 &2e_2&0\\
      0 & 2e_3 &e_1+2e_4&3e_2\\
      0 & 0 &e_3&3e_4
    \end{matrix}
  \right).
\]
The important point for our computations is that when $\alpha$ is two-dimensional,
\begin{equation}
  \label{eq:symmexps}
  \Tr(S^3L) = 6\Tr(L).
\end{equation}
The association $A \mapsto \Sym^3 A$ defines a homogeneous polynomial map of degree $3$
\[
  M_{k}(\alpha,L) \to M_{3k}(\Sym^3 \alpha,S^3L).
\]
This map is not quite injective, it is not linear, and there is no reason, in general, why it needs to be onto. The next theorem says that, nevertheless, for $\alpha$ of rank $2$ and $A$ a minimal weight form for $(\alpha,L)$, the symmetric cube lift $F =\Sym^3A$ is a minimal weight form for $(\Sym^3\alpha, S^3L)$ that generates $M(\Sym^3\alpha,S^3L)$ as a cyclic $M(1)\langle D\rangle$-module.
\begin{thm}
  \label{t:symmetriccube}
  Let $\alpha$ be an irreducible representation of $\Gamma$ of rank $2$ such that $\Sym^3\alpha$ is irreducible, let $L$ denote a choice of exponents for $\alpha(T) = e^{2\pi iL}$, and let $S^3L$ denote the corresponding symmetric cube exponents. Let $k_1 = 6\Tr(L)-1$ denote the minimal weight for $\alpha$, and let $A  = (f,g)^T \in M_{k_1}(\alpha,L)$ denote a form of minimal weight for $\alpha$ with respect to the choice of exponents $L$. Then the following hold:
  \begin{enumerate}
  \item the $M(1)\langle D\rangle$-module $M(\Sym^3\alpha,S^3L)$ is cyclic;
  \item the minimal weight for $M(\Sym^3\alpha,S^3L)$ is $3k_1$;
  \item if
      \[
  F = \Sym^3A = \left(\begin{matrix}
f^3\\ f^2g\\ fg^2\\ g^3
    \end{matrix}
  \right),
\]
then an $M(1)$-basis for $M(\Sym^3\alpha,S^3L)$ is given by $F$, $DF$, $D^2F$, $D^3F$.
  \end{enumerate}
\end{thm}

\begin{proof}
  Set $\rho = \Sym^3\alpha$. To see (1), we apply Theorem \ref{t:weights}. Let $u$ and $v$ denote the eigenvalues of $L$. By  the classification of two-dimensional irreducible representations of $\Gamma$ from Section \ref{s:vvmfs}, we find that $\Tr(S^3L) = 6\Tr(L) \equiv a \pmod{6}$, where $\det \alpha(T) = e^{2\pi i a/6}$ and $\alpha(-1) = (-1)^{a+1}$. Note that then also $\rho(-1) = (-1)^{a+1}$, so that $3\Tr(S^3L)$ is not congruent to the sign of $\rho$ mod $2$. Thus Theorem \ref{t:weights} implies that the module $M(\Sym^3\alpha,S^3L)$ is cyclic.

  Now that we know we are in the cyclic case, Theorem \ref{t:weights} gives the minimal weight as $3\Tr(S^3L) - 3 = 18\Tr(L)-3 = 3k_1$, which proves (2).

Finally, (3) follows since $M(\Sym^3\alpha,S^3L)$ is cyclic by (1), and $F$ is a nonzero form of minimal weight by (2).
\end{proof}

Unfortunately, as in the case of the tensor product, the exponents $S^3L$ are not typically the canonical exponents for $\Sym^3\alpha$, and so $M(\Sym^3\alpha,S^3L)$ does not always describe all holomorphic forms for $\Sym^3\alpha$. However, every lattice in $M^\dagger(\rho)$ is indeed contained in \emph{some} $M(\Sym^3\alpha,S^3L)$, and so in this sense Theorem \ref{t:symmetriccube} is reasonably complete. Note too that the form $A$, and hence also $F$, can be made explicit using Theorem \ref{t:rank2}. 

\begin{ex}
A simple example demonstrating that symmetric power exponents $S^nL$ are not always canonical, even if $L$ is canonical, occurs already in rank one. Let $\chi$ be the character of $\eta^2$, and let $L = 1/12$ be the canonical choice of exponents for $\chi$. Then $\Delta= (\eta^{2})^{12}$ is a form for $\Sym^{12}\chi = \chi^{12} = 1$ with respect to the exponents $S^{12}L = 1$. This is different from the canonical choice of exponent for the trivial representation, which is $0$. In this case $M(\Sym^{12}\chi,S^{12}L) \subseteq M(1)$ is the $M(1)$-submodule of cusp forms. This submodule is spanned by $\Delta = \Sym^{12}\eta^{2}$ as an $M(1)$-module, and it does not contain all holomorphic forms for $\Sym^{12}\chi = 1$.
\end{ex}

\section{Induction of two-dimensional representations}
\label{s:induction}

Let $G \subseteq \SL_2(\ZZ)$ be the unique (normal) subgroup of index $2$. The nontrivial coset is represented by $S = \stwomat 0{-1}10$. Set
\begin{align*}
R_0 &= ST = \twomat{0}{-1}{1}{1},\\
R_1 &= TR_0T^{-1} = TS = \twomat{1}{-1}{1}{0}.
\end{align*}
Then $G$ is generated by these matrices subject to the relations $R_0^3 = R_1^3 = -1$.

Let $\rho$ be an irreducible representation of $G$. Since $-1$ is in the center of $G$, necessarily $\rho(-1) = \pm 1$. Hence $\rho$ is uniquely determined by a parity $e \pmod{2}$ such that $\rho(-1) = (-1)^e$ and matrices $\rho(R_0)$ and $\rho(R_1)$ satisfying $\rho(R_0)^3 = \rho(R_1)^3 = (-1)^e$.

Assume that $\rho$ is of rank $2$. We may diagonalize $\rho(R_0)$. Its eigenvalues are then sixth roots of unity, and they must be distinct, for otherwise $\rho(R_0)$ would be diagonal and this would contradict the irreducibility of $\rho$. Thus write
\begin{align*}
  \rho(R_0) &= (-1)^e\twomat {\zeta_1}{0}{0}{\zeta_2}, & \rho(R_1)& = (-1)^e\twomat abcd,
\end{align*}
where $\zeta_1^3 = \zeta_2^3 = 1$, but $\zeta_1\neq \zeta_2$. If $b = 0$ then observe that the column vector $(0,1)^T$ spans a subrepresentation. Hence $b \neq 0$. Likewise, $c \neq 0$. At this point, without disturbing the diagonalization of $\rho(R_0)$, the only freedom we have in changing basis is to rescale basis vectors. This corresponds to a conjugation
\[
  \twomat{u}{0}{0}{v}\twomat abcd \twomat {u^{-1}}00{v^{-1}} = \twomat{a}{(u/v)b}{(v/u)c}{d}.
\]
Thus our last degree of freedom allows us to assume that $b = 1$. The moduli space of representations will be described by the equations arising from the condition $\rho(R_1)^3 = \veps I$. This condition is equivalent to the following two conditions:
\begin{align*}
  c &= -a^2-ad-d^2,\\
  -1 &= (a+d)^3.
\end{align*}
Write $a+d = -\zeta_3$ where $\zeta_3^3 = 1$, and this choice of $\zeta_3$ is any of the three possibilities. Hence we can write $d = -\zeta_3-a$. It follows that
\[
  -a^2-ad-d^2 = -a^2+\zeta_3a+a^2-(\zeta_3+a)^2 = -(a^2+\zeta_3a+\zeta_3^2).
\]
Thus, two dimensional irreducible representations of $G$ are classified by a choice of $e$, three third roots of unity $\zeta_1$, $\zeta_2$ and $\zeta_3$ but with $\zeta_1 \neq \zeta_2$, and a free parameter $a$ subject to the condition $a^2+\zeta_3a+\zeta_3^2 \neq 0$. We summarize and expand on these results in the following Proposition.
\begin{prop}
\label{p:irrepsofG}
  Let $G {\subseteq} \SL_2(\ZZ)$ denote the unique normal subgroup of index two. Then the following properties hold.
\begin{enumerate}
  \item Every irreducible representation of $G$ of rank $2$ is isomorphic to a representation of the form $\rho = \rho(e,\zeta_1,\zeta_2,\zeta_3,a)$ characterized by $\rho(-1) = (-1)^e$,
\begin{align*}
\rho(R_0) &= (-1)^e\twomat{\zeta_1}00{\zeta_2}, & \rho(R_1) &= (-1)^e\twomat{a}{1}{-a^2-\zeta_3a-\zeta_3^2}{-\zeta_3-a},
\end{align*}
where $\veps = \pm 1$, $\zeta_1^3=\zeta_2^3 =\zeta_3^3= 1$, $\zeta_1\neq \zeta_2$, and $a$ is a free parameter such that $a^2+\zeta_3a+\zeta_3^2 \neq 0$. There are isomorphisms
\[
  \rho(e,\zeta_1,\zeta_2,\zeta_3,a) \cong \rho(e,\zeta_2,\zeta_1,\zeta_3,-\zeta_3-a),
\]
but otherwise no two such distinct representations $\rho$ are isomorphic.

\item An irreducible representation $\rho(e,\zeta_1,\zeta_2,\zeta_3,a)$ is the restriction of a representation of $\SL_2(\ZZ)$ if and only if $\zeta_1+\zeta_2+\zeta_3 = 0$.

\item If $\rho = \rho(e,\zeta_1,\zeta_2,\zeta_3,a)$, then the induced representation $\Ind_{G}^{\Gamma}\rho$ is irreducible if and only if $\zeta_1+\zeta_2+\zeta_3 \neq 0$ and $a \neq (-1)^e\frac{\zeta_1\zeta_2+\zeta_2\zeta_3+\zeta_3^2}{\zeta_1-\zeta_2}$.
\end{enumerate}
\end{prop}
\begin{proof}
  We have seen above that each irreducible representation of rank $2$ is isomorphic to a $\rho$ as in the statement of the Proposition. From the discussion above it is likewise clear that $\rho$ is unique up to the permutation of $\zeta_1$ and $\zeta_2$. A straightforward computation shows that conjugation by 
  \[
\twomat 01{-(a^2+\zeta_3a+\zeta_3^2)}0
  \]
  realizes the isomorphism $\rho(e,\zeta_1,\zeta_2,\zeta_3,a) \cong \rho(e,\zeta_2,\zeta_1,\zeta_3,-\zeta_3-a)$.

  For claim (2), let $\rho$ denote an irreducible representation of $\SL_2(\ZZ)$ of rank $2$. Similarly to above, we can write
\begin{align*}
\rho(-1) &= (-1)^e, & \rho(R_0) &= (-1)^e\twomat{\zeta_1}00{\zeta_2}, & \rho(S) &= \twomat{x}{1}{y}{z}
\end{align*}
where $\rho(S)^2 = (-1)^e$. Hence
\[
  (-1)^e = \twomat{x^2+y}{x+z}{(x+z)y}{y+z^2}.
\]
Therefore  $x+z = 0$ and $x^2+y = y+z^2 = \veps$, and thus
\[
\rho(S) = \twomat{x}{1}{(-1)^e-x^2}{-x}
\]
Observe that $R_1 = S^{-1}R_0S$ and thus
\[
  \rho(R_1) = \twomat{(\zeta_1-\zeta_2)x^2+\veps\zeta_2}{(\zeta_1-\zeta_2)x}{(\zeta_1-\zeta_2)x(\veps-x^2)}{-(\zeta_1-\zeta_2)x^2+\veps\zeta_1}
\]
Conjugate by $\stwomat{(\zeta_1-\zeta_2)^{-1}x^{-1}}{0}{0}{1}$ and the representation satisfies $\rho(-1) = (-1)^e I$,
\begin{align*}
\rho(R_0) &= (-1)^e\twomat{\zeta_1}00{\zeta_2}, & \rho(R_1) &= \twomat{(\zeta_1-\zeta_2)x^2+(-1)^e\zeta_2}{1}{(\zeta_1-\zeta_2)^2x^2((-1)^e-x^2)}{-(\zeta_1-\zeta_2) x^2+(-1)^e\zeta_1}.
\end{align*}
Thus, if a two dimensional irreducible representation $\rho(e, \zeta_1,\zeta_2,\zeta_3,a)$ of $G$ is the restriction of a representation of $\SL_2(\ZZ)$, we must have $\zeta_1+\zeta_2+\zeta_3=0$. Conversely, if this identity holds for a representation $\rho$ of $G$, then the values of $x$ satisfying 
\[
x^2 = (-1)^e\frac{ a-\zeta_2}{\zeta_1-\zeta_2}
\]
allow us to define $\rho(S)$ compatibly as above, and extend $\rho$ to $\Gamma$. This verifies claim (2).

Finally, we treat claim (3). If $\rho$ is the restriction of a representation of $\Gamma$, then the universal property of the induction yields a surjective map $\Ind_{G}^{\Gamma} \rho \to \rho$, whose kernel is a two dimensional subrepresentation of $\Ind \rho \df \Ind_G^{\Gamma}\rho$. Hence $\Ind\rho$ is not irreducible. Thus, by (2), if $\zeta_1+\zeta_2+\zeta_3 = 0$, then $\Ind \rho$ is not irreducible.

It is more difficult to characterize when $\Ind \rho$ is irreducible, since we are working with infinite discrete groups, and thus we don't have access to standard techniques such as Mackey theory. We will perform explicit computations with bases: since $S$ represents the nontrivial coset of $G$ in $\SL_2(\ZZ)$, up to isomorphism we have $(\Ind \rho)(-1) = (-1)^e$ and
\begin{align*}
(\Ind\rho)(R) &= \veps\left(\begin{matrix}
    \zeta_1&0&0&0\\
    0&\zeta_2&0&0\\
    0&0&a&1\\
    0&0&-a^2-\zeta_3a-\zeta_3^2&-\zeta_3-a
  \end{matrix}\right),&  (\Ind\rho)(S) &= \left(\begin{matrix}
    0&0&(-1)^e&0\\
    0&0&0&(-1)^e\\
    1&0&0&0\\
    0&1&0&0
  \end{matrix}\right).
\end{align*}

For simplicity first assume $e = 0$. Then there exists a basis for $\Ind \rho$ in which $(\Ind \rho)(S) = \diag(1,-1,1,-1)$ is diagonal, and  such that the matrix of $(\Ind \rho)(R)$ has the form
\[
  (\Ind \rho)(R) = \frac 12 \left(\begin{matrix}
      a+\zeta_1 & -a+\zeta_1 &1& -1\\
      -a+\zeta_1 & a+\zeta_1 &-1& 1\\
      -a^2-a\zeta_3-\zeta_3^2 & a^2+a\zeta_3+\zeta_3^2 & -a+\zeta_2-\zeta_3 & a+\zeta_2+\zeta_3\\
      a^2+a\zeta_3+\zeta_3^2 & -a^2-a\zeta_3-\zeta_3^2 & a+\zeta_2+\zeta_3 & -a+\zeta_2-\zeta_3
    \end{matrix}
  \right)
\]

First consider when $\Ind \rho$ could have a one dimensional subrepresentation. If $e_1$, $e_2$, $e_3$ and $e_4$ are the eigenvectors corresponding to the diagonalization of $(\Ind \rho)(S)$, then the subrepresentation must be spanned by a vector $e_3$, $e_4$, $e_1+ue_3$,  or $e_2+ve_4$. It's easy to see that $e_3$ and $e_4$ don't span subrepresentations. Observe that
\begin{align*}
  (\Ind \rho)(R)\left(\begin{matrix}
      1\\
      0\\
      u\\
      0 
  \end{matrix}\right) &= \frac 12\left(\begin{matrix}
      a+\zeta_1+u\\
      -a+\zeta_1-u\\
      -a^2-a\zeta_3-\zeta_3^2-au+\zeta_2u-\zeta_3u\\
      a^2+a\zeta_3+\zeta_3^2+au+\zeta_2u+\zeta_3u 
    \end{matrix}\right)\\
    (\Ind \rho)(R)\left(\begin{matrix}
      0\\
      1\\
      0\\
      v 
  \end{matrix}\right) &= \frac{1}{2}\left(\begin{matrix}
-a+\zeta_1-v\\ a+\zeta_1+v\\ a^2+a\zeta_3+\zeta_3^2+av+\zeta_2v+\zeta_3v\\ -a^2-a\zeta_3-\zeta_3^2-av+\zeta_2v-\zeta_3v
  \end{matrix}\right).
\end{align*}
If the vector $e_1+ue_3$ spans a subrepresentation, then the second and last coordinates above must be zero. It follows that $u = \zeta_1-a$ and
\begin{align*}
  (\Ind \rho)(R)\left(\begin{matrix}
      1\\
      0\\
      \zeta_1-a\\
      0 
  \end{matrix}\right) &= \left(\begin{matrix}
      \zeta_1\\
      0\\
      \zeta_2(\zeta_1-a)\\
      0 
  \end{matrix}\right).
\end{align*}
Since $\zeta_1\neq \zeta_2$, this shows that $e_1+ue_3$ does not span a subrepresentation, as it it not an eigenvector for $(\Ind \rho)(R)$. A similar argument applies to show that $e_2+ve_4$ does not span a subrepresentation of $\Ind \rho$ for any choice of scalar $u$, and thus $\Ind \rho$ does not contain any one dimensional subrepresentations. By duality, if $\rho$ is irreducible then $\Ind \rho$ never contains a three dimensional subrepresentation either.

Thus, we are reduced to considering when $\Ind \rho$ contains a two dimesional irreducible subrepresentation $\rho'$. In this case $\rho' (S)$ must have $1$ and $-1$ as eigenvalues, so that $\rho'$ is spanned by vectors of the form: $\{e_3, e_4\}$, $\{e_3,e_2+ve_4\}$, $\{e_1+ue_3,e_4\}$ or $\{e_1+ue_3,e_2+ve_4\}$ for complex scalars $u$ and $v$. Since the upper right block of $(\Ind\rho)(R)$ contains nonzero constants, the span of $\{e_3,e_4\}$ does not define a subrepresentation. The spans of $\{e_3, e_2+ve_4\}$ and $\{e_1+ue_3,e_4\}$ correspond to vectors with $e_1$ coordinate, respectively $e_2$ coordinate, equal to zero. These are likewise easily seen not to be stable under $\Ind \rho$.

Thus, we must determine when there exist complex scalars $u$ and $v$ such that the span $V$ of $\{e_1+ue_3,e_2+ve_4\}$ is stable under $\Ind\rho$. In order for $V$ to be stable under $\Ind \rho$, $u$ and $v$ must satisfy the equations:
\begin{align*}
  u&=\frac{-a^2-a\zeta_3-\zeta_3^2-au+\zeta_2u-\zeta_3u}{a+\zeta_1+u}, & u &= \frac{a^2+a\zeta_3+\zeta_3^2+av+\zeta_2v+\zeta_3v}{-a+\zeta_1-v},\\
  v&=\frac{a^2+a\zeta_3+\zeta_3^2+au+\zeta_2u+\zeta_3u}{-a+\zeta_1-u}, & v &= \frac{-a^2-a\zeta_3-\zeta_3^2-av+\zeta_2v-\zeta_3v}{a+\zeta_1+v}.
\end{align*}
Let $I \subseteq \CC[\zeta_1,\zeta_2,\zeta_3,a,u,v]$ denote the ideal generated by the relations above after the denominators have been cleared, and treating $\zeta_1$, $\zeta_2$ and $\zeta_3$ as formal variables. A Groebner basis computation reveals that $(\zeta_1-\zeta_2)(\zeta_1+\zeta_2+\zeta_3)(u-v) \in I$.

If $\zeta_1+\zeta_2+\zeta_3 \neq 0$ then $u =v$. In this case a Groebner basis computation then shows that $u(u+a-\zeta_2) = 0$. But we can't have $u = v= 0$, for this would mean that the span of $e_1$ and $e_2$ is stable under $(\Ind \rho)(R)$, and it clearly is not. Hence $u = v= \zeta_2-a$, and a final Groebner basis computation implies that
\[
  a = \frac{\zeta_1\zeta_2+\zeta_2\zeta_3+\zeta_3^2}{\zeta_1-\zeta_2}.
\]
It is now straightforward to verify that this choice of $a$ does indeed yield an induced representation that is not irreducible.

The other case is that $\zeta_1+\zeta_2+\zeta_3 = 0$, and we have already seen that $\Ind \rho$ is not irreducible in this case, since then $\rho$ is the restriction of a representation of $\SL_2(\ZZ)$.

This concludes the proof of (3) when $e = 0$. If $e = 1$ then we can reduce to the case of $e = 0$ by tensoring with a one dimensional representation $\chi$ such that $\chi(-I) = -1$. Notice that then the condition $a\neq \frac{\zeta_1\zeta_2+\zeta_2\zeta_3+\zeta_3^2}{\zeta_1-\zeta_2}$ is replaced by $a \neq - \frac{\zeta_1\zeta_2+\zeta_2\zeta_3+\zeta_3^2}{\zeta_1-\zeta_2}$. This concludes the proof.
\end{proof}


Now we use notation as in Section 6.2 of \cite{CandeloriFranc2}. Recall from \cite{CandeloriFranc2} that $\beta$ denotes the character of $G$ satisfying $\beta(-1) =1$, $\beta(R_0) = \zeta$, $\beta(R_1) = \zeta^2$ where $\zeta = e^{2\pi i/3}$. This character $\beta$ is not the restriction of a character of $\SL_2(\ZZ)$. Observe that
\begin{equation}
  \label{eq:tensorbybeta}
  \rho(e,\zeta_1,\zeta_2,\zeta_3,a) \otimes \beta \cong \rho(e,\zeta_1\zeta,\zeta_2\zeta,\zeta_3\zeta^2,a\zeta^2).
\end{equation}
\begin{cor}
  \label{c:restriction}
Let $\rho \colon G \to \GL_2(\CC)$ be a representation such that $\Ind_G^\Gamma \rho$ is irreducible. Then exactly one of the representations $\rho$, $\rho\otimes \beta$ and $\rho \otimes \beta^2$ is the restriction of a representation from $\Gamma$.
\end{cor}
\begin{proof}
Note that $\rho$ is irreducible, so that $\rho \cong \rho(e,\zeta_1,\zeta_2,\zeta_3,a)$ for some choice of parameters, by Proposition \ref{p:irrepsofG}. The condition that $\zeta_1+\zeta_2+\zeta_3 = 0$ from Proposition \ref{p:irrepsofG} is equivalent to $\zeta_1$, $\zeta_2$ and $\zeta_3$ being the three distinct cube roots of unity. Thus, it follows from Equation \eqref{eq:tensorbybeta} and Proposition \ref{p:irrepsofG} that, if $\Ind \rho$ is irreducible, then exactly one representation from the three $\rho$, $\rho \otimes \beta$, $\rho \otimes \beta^2$ is the restriction of a representation from $\Gamma$.
\end{proof}

The character $\beta$ satisfies $\beta(T^2) = 1$, and it generates the subgroup of $\Hom(G,\CC^\times)$ of characters that are trivial on $T^2$. Such characters are called \emph{cuspidal}. They are the characters of the fundamental group of the compact modular curve associated to $G$. Following \cite{CandeloriFranc2}, if $L$ denotes a choice of exponents for $\rho(T^2)$, then it also defines a choice of exponents for $(\rho \otimes \beta)(T^2)$ and $(\rho \otimes \beta^2)(T^2)$. The corresponding space of \emph{geometrically weighted modular forms} with growth condition at the cusp of type $L$ is defined to be
\[
\GM(G,\rho,L) \df M(G,\rho,L)\oplus M(G,\rho\otimes \beta,L) \oplus M(G,\rho \otimes \beta^2,L). 
\]
We will include the groups $G$ and $\Gamma$ in the notation now, as both group swill play a role. The module $\GM(G,\rho,L)$ has a grading of type $\ZZ\oplus (\ZZ/3\ZZ)$ where the first factor $\ZZ$ corresponds to the weight, while the second factor $\ZZ/3\ZZ$ corresponds to the power of $\beta$ occuring in the transformation law. Let $S(G) \df \GM(G,1)$ denotes the ring of geometrically weighted modular forms for $G$ with respect to a canonical choice of logarithm. Recall that the following variant of the free-module theorem for $\SL_2(\ZZ)$ holds for $G$:
\begin{thm}
\label{t:FMTforG}
  If $\rho$ is a representation of $G$ and if $L$ denotes a choice of exponents for $\rho(T^2)$, then the module $\GM(G,\rho,L)$ is a free $(\ZZ\times \ZZ/3\ZZ)$-graded module of rank $\dim_{\CC} \rho$ over the ring $S(G)$ of geometrically weighted modular forms for $G$.
\end{thm}
\begin{proof}
This follows from Corollary 4.8 of \cite{CandeloriFranc2}. The key point is that $G$ only has one cusp and two elliptic points.
\end{proof}

The ring $S(G)$ is polynomial in two generators in weight $2$, and these generators can be described in terms of classical theta series. Recall from \cite{CandeloriFranc2} that if
\begin{align*}
  f &= (1+e^{2\pi i \frac{1}{6}})\theta_2^4-e^{2\pi i \frac{5}{6}}(\theta_3^4+\theta_4^4),\\
  g &= f|T,
\end{align*}
then $f \in M_2(G,\beta)$, $g \in M_2(G,\beta^2)$ and $S(G) = \CC[f,g]$ as $\ZZ\times (\ZZ/3\ZZ)$-graded rings. It's not too hard to show that $S(G) = M(\Gamma(2))$. More generally we can prove the following:
\begin{lem}
  \label{l:level2}
  Let $\rho$ be a representation of $G$ and let $L$ denote a choice of exponents for $\rho$. Then $L' = (L, \rho(R_0)L\rho(R_0^{-1}),\rho(R_0^2)L\rho(R_0^{-2}))$ denotes a choice of exponents for $\rho|_{\Gamma(2)}$ and we have the following:
  \begin{enumerate}
  \item $\Ind_{\Gamma(2)}^G (\rho|_{\Gamma(2)}) \cong \rho  \otimes (1\oplus \beta \oplus \beta^2)$;
  \item $\GM_{k}(G,\rho,L)$ is the space of global sections of the bundle $\pi_*\pi^*\cV_{k}(\rho,L)$ where $\pi \colon X(\Gamma(2)) \to X(G)$ is the natural map between compact orbifolds;
  \item there is a natural identification $\GM(G,\rho,L) = M(\Gamma(2),\rho|_{\Gamma(2)},L')$;
    \item in particular, $S(G) = M(\Gamma(2))$.
\end{enumerate}
\end{lem}
\begin{proof}
  For part (1), note that quite generally, restricting a representation to a finite index subgroup and then inducing back is the same as tensoring with the permutation representation given by the cosets. In this case the permutation representation on $G/\Gamma(2)$ is isomorphic with $1 \oplus \beta \oplus \beta^2$, which proves (1).

To prove (2), first note that $\pi^*\cV_{k}(G,\rho,L) \cong \cV_{k}(\Gamma(2),\rho|_{\Gamma(2)},L')$, as one sees by examining the behaviour at the three cusps of $\Gamma(2)$, which are conjugate in $G$ by the powers of $R_0$. Therefore, we wish to describe the bundle
\[
  \pi_*\pi^*\cV_{k}(G,\rho,L) \cong \pi_*\cV_{k}(\Gamma(2),\rho|_{\Gamma(2)},L').
\]
By (1), over the open curve associated to $G$, this bundle is isomorphic with $\cV(G,\rho\otimes (1 \oplus \beta \oplus \beta^2))$. If we work in a basis for $\Ind_{\Gamma(2)}^G(\rho|_{\Gamma(2)})$ such that the image of $T$ is diagonal, then the matrix $\diag(L,L,L)$ is a possible choice of exponents. Using this basis, and since $\beta(T^2) = \beta^2(T^2) = 1$, an explicit computation at the cusp shows that 
\[
  \pi_*\pi^*\cV_{k}(G,\rho,L) \cong \cV_{k}(G,\Ind_{\Gamma(2)}^G(\rho|_{\Gamma(2)}),\diag(L,L,L))
\]
and therefore
\begin{align*}
  H^0(\pi_*\pi^*\cV_{k}(G,\rho,L)) &\cong H^0\left(\cV_{k}\left(G,\rho \otimes(1\oplus\beta\oplus\beta^2),L\oplus L\oplus L\right)\right)\\
  &\cong \GM_{k}(G,\rho,L).
\end{align*}
This proves (2), and (3) is then deduced from (2) by projecting to the first $\dim \rho$ components.

  Part (4) follows immediately from part (3) by taking $\rho = 1$ and $L = 0$, the canonical choice of exponents for $\rho$. Then $L' = (0,0,0)$, so that $L'$ is the canonical choice of exponents for the trivial representation of $\Gamma(2)$. Therefore
  \[
  S(G) = \GM(G,1,L) = M(\Gamma(2),1,L') = M(\Gamma(2)).
  \]
\end{proof}
The additional grading by $\ZZ/3\ZZ$ on $S(G)$ corresponds to breaking $M(\Gamma(2))$ up into $G$-isotypic components for the induced action of $G/\Gamma(2) \cong \ZZ/3\ZZ$. Hence $f,g \in M(\Gamma(2))$ are $G$-semiinvariants that generate $M(\Gamma(2))$ as a $\CC$-algebra.
\begin{rmk}
In \cite{CandeloriFranc2}, Theorem \ref{t:FMTforG} was deduced using general geometric machinery for genus zero orbifolds. One can instead use Lemmma \ref{l:level2} to deduce Theorem \ref{t:FMTforG} from the usual splitting principle for vector bundles on $\PP^1$ by identifying the modular curve $X(\Gamma(2))$ with the projective line. One technical point is that to identify $X(\Gamma(2))$ with $\PP^1$ we should work inside $\PSL_2(\ZZ)$, rather than in $\SL_2(\ZZ)$. Likewise, the paper \cite{CandeloriFranc2} assumes that $G \subseteq \PSL_2(\ZZ)$. However, since $-1$ is contained in our groups, we can reduce to working in $\PSL_2(\ZZ)$ by breaking representations into even and odd parts, and then replacing the odd part by a twist (and adjusting exponents suitably). Hence Theorem \ref{t:FMTforG} indeed holds as stated for $G \subseteq \SL_2(\ZZ)$.
\end{rmk}

\begin{prop}
  \label{p:inductioncase1}
  Let $\{\rho,\rho\otimes \beta, \rho\otimes \beta^2\}$ denote an orbit of two-dimensional irreducible representations of $G$ under the action of the cuspidal characters, normalized so that $\rho$ is the restriction of a representation of $\SL_2(\ZZ)$. Let $L$ denote a choice of exponents for $\rho(T^2)$. Then one of the following holds: if $k_i$ denotes the minimal classical weight for $\rho\otimes \beta^i$ with respect to the exponents $L$, then
\begin{enumerate}
\item $k_1 = k_2 = k_0+2$. In this case $\GM(G,\rho,L)$ is generated by a nonzero form $A \in M_{k_0}(G,\rho,L)$ and its derivative $DA\in M_{k_0+2}(G,\rho,L)$ as a free-module over $S(G)$. The form $A$ satisfies a differential equation of the form
  \[
  D^2A + \alpha E_4A = 0.
\]
    \item $k_1=k_2$ and $k_0 = k_1+2$. In this case there is a free basis $A \in M_{k_1}(G,\rho\otimes \beta,L)$ and $B\in M_{k_1}(G,\rho\otimes \beta^{2},L)$ for $\GM(G,\rho,L)$ over $S(G)$ such that
    \begin{align*}
  D(A,B) = (A,B)\twomat{0}{u f}{g}{0}
    \end{align*}
    for $u \in \CC^\times$.
  \end{enumerate}
\end{prop}
\begin{proof}
The usual modular differential operator $D$ defines a map in weight $k$
\[
  D_k \colon M_{k}(G,\rho \otimes \beta^j,L) \to M_{k+2}(G, \rho\otimes \beta^j,L),
\]
and thus it acts on geometric weights $(k,\beta^j)$ by increasing the classical weight $k$ by $2$ and leaving the geometric $\ZZ/3\ZZ$ component $\beta^j$ fixed. Let $A$, $B \in \GM(G,\rho,L)$ denote forms of minimal classical weights $a$ and $b$, respectively, that are a basis for $\GM(G,\rho,L)$ as a graded module over $S(G)$. Assume without loss of generality that $a \leq b$, and $A \in M_{a}(G,\rho\otimes \beta^i,L)$ and $B \in M_{b}(G,\rho \otimes \beta^j,L)$ for some integers $i,j \in \{0,1,2\}$.

First suppose $a \neq b$, so that $b \geq a+2$. Observe that $DA \in M_{a+2}(G,\rho\otimes \beta^i,L)$, and $DA$ must be nonzero by irreducibility of $\rho$. Since the weight $(2,\beta^0)$-component of $S(G)$ is zero, it follows that $DA$ is not an $S(G)$-multiple of $A$. Hence $A,DA$ must be a basis for $\GM(G,\rho,L)$ over $S(G)$ and so, up to rescaling, $B = DA$. In this case $\GM(G,\rho,L)$ is cyclic as an $S(G)\langle D\rangle$-module. Since the $(4,\beta^0)$-component of $S(G)$ is spanned by $E_4$, it follows that $A$ satisfies a modular linear differential equation of the form
\begin{equation}
  \label{eq:inducedMLDE1}
  D^2A + \alpha E_4A = 0
\end{equation}
for some $\alpha \in \CC$. Since $E_4$ is in fact a modular form for $\Gamma$, it follows that for all $\gamma \in \Gamma$,
\[
  D^2(A|_{a}\gamma) + \alpha E_4(A|_{a}\gamma) = 0.
\]
It follows from this that the irreducible representation $\rho\otimes \beta^i$ can be extended to $\Gamma$. Thus, thanks to our normalization of $\{\rho, \rho\otimes \beta, \rho\otimes \beta^2\}$, we have $i = 0$. This treats case (1).

Next suppose that $a = b$. As in the previous paragraph, it can't be that $B \in M_{a}(G,\rho\otimes \beta^i,L)$, for otherwise it would be impossible to express $DA$ and $DB$ as a linear combination of $A$ and $B$ over $S(G)$. If $B \in M_{a}(G,\rho\otimes \beta^{i+1},L)$ then we must have $DA = v gB$ and $DB = u fA$ for nonzero scalars $u$ and $v$. Thus, after rescaling, in this case we can find a basis $A \in M_{a}(G,\rho\otimes \beta^i,L)$, $B \in M_{a}(G,\rho\otimes \beta^{i+1},L)$ such that
\begin{align*}
  D(A,B) = (A,B)\twomat{0}{u f}{g}{0}
\end{align*}
for a nonzero scalar $u \in \CC$. In the final case, we may similarly find a free basis $A \in M_{a}(G,\rho\otimes \beta^i,L)$ and $B \in M_{a}(G,\rho\otimes \beta^{i+2},L)$ such that
\begin{align*}
  D(A,B) = (A,B)\twomat{0}{u g}{f}{0}.
\end{align*}
In this case we can replace $\beta^i$ by $\beta^{i+2}$ to reduce to the previous case. This treats case (2).
\end{proof}

  In both cases of Proposition \ref{p:inductioncase1}, the minimal weights for $\rho \otimes \beta$ and $\rho \otimes \beta^2$ are equal, and the spaces of minimal weight forms are one dimensional. If we let $A$ and $B$ denote bases for the minimal weight spaces $M_{k_1}(G,\rho\otimes \beta,L)$ and $M_{k_1}(G,\rho\otimes \beta^2,L)$, respectively, then the analysis of part (2) of Proposition \ref{p:inductioncase1} still applies to $A$ and $B$, and they satisfy a differential equation as in part (2) of Proposition \ref{p:inductioncase1}. Therefore, in all cases, if we normalize $\{\rho, \rho\otimes \beta, \rho \otimes \beta^2\}$ so that $\rho$ is the restriction of a representation of $\Gamma$, then we can find minimal weight forms for $\rho\otimes \beta$ and $\rho \otimes \beta^2$ by solving a matrix ordinary differential equation of the form
\begin{equation}
    \label{eq:matrixmlde}
  D(A,B) = (A,B) \twomat{0}{uf}{g}{0}.
\end{equation}
In order use this equation, we will need to determine the minimal (classical) weight $k_1$ for $\rho \otimes \beta$ and $\rho \otimes \beta^2$, and use an appropriate modular function of weight zero to turn the matrix equation \eqref{eq:matrixmlde} into an algebraic differential equation on the modular curve defined by $G$. We attack the latter problem first, and then come back to the question of the minimal weight.
  
To begin, set
\[
  h = \frac{E_6}{12^{3/2}\eta^{12}} = \frac{E_6}{\sqrt{E_4^3-E_6^2}}.
\]
This modular form has weight zero, but it does not map the elliptic points and cusps of $G$ to $0$, $1$ and $\infty$. Observe that since $K = 1728/j = (E_4^3-E_6^2)/E_4^3$ we have
\[
  h^2+1 = \frac{E_6^2}{E_4^3-E_6^2}+1 = K^{-1}.
\]
Note that $K(\infty) = 0$, $K(i)=1$ and  $K(\zeta) = \infty$. Therefore $h(\infty) = \infty$, $h(i) = 0$ and $h(\zeta) = \pm i$. We want to choose a hauptmodul $Z$ satisfying $Z(\infty) = 0$, $Z(\zeta) = \infty$ and $Z(\zeta+1) = 1$. Classical results on values of the Dedekind $\eta$-function at quadratic irrationalities allow one to show that $h(\zeta) = i$ and $h(\zeta+1) = -i$. Therefore, set
\[
  Z = \frac{2}{1+ih}.
\]
Then we have $Z(\infty) = 0$, $Z(\zeta) = \infty$, $Z(\zeta+1) = 1$ and $Z(i) = 2$ (although it is not necessary to know this last value). Note also that
\[
  K^{-1} = 1-\left(\frac{2}{Z}-1\right)^2=\frac{Z^2-(2-Z)^2}{Z^2}=\frac{-4+4Z}{Z^2}
\]
so that $K = \frac{Z^2}{4(Z-1)}$. This value of $Z$ is our hauptmodul, or uniformizer, for $G$:
\begin{equation}
  \label{eq:Ghauptmodul}
 Z =2\frac{12^{3/2} \eta^{12}}{12^{3/2}\eta^{12}+iE_6}.
\end{equation}

We will require a number of other identities. First are the relations:
\begin{align*}
fg &= -4e^{2\pi i/6}E_4, & f^3+g^3 &= 16E_6, & Z &= \frac{f^3-g^3}{f^3}.
\end{align*}
Next let $\theta_Z = Z\frac{d}{dZ}$. We have the identity
\[
  dK = \frac{8Z(Z-1)-4Z^2}{16(Z-1)^2}dZ = \frac{Z^2-2Z}{4(Z-1)^2}dZ
\]
Therefore,
\[
  \theta_K = K\frac{d}{dK} = \frac{Z^2}{4(Z-1)}\frac{4(Z-1)^2}{Z^2-2Z}\frac{d}{dZ}=\frac{Z-1}{Z-2}\theta_Z.
\]
Since $\theta_q = \frac{E_6}{E_4}\theta_K$ we deduce that
\[
  \theta_q = \frac{-e^{2\pi i/6}(f^3+g^3)}{4fg}\frac{Z-1}{Z-2}\theta_Z
\]
It turns out that
\[
  \frac{-e^{2\pi i/6}(f^3+g^3)}{4fg}\frac{Z-1}{Z-2} = \frac{g^2}{4(e^{2\pi i/6}-1)f}
\]
Thus, our analogue of the identity $\theta_q = \frac{E_6}{E_4}\theta_K$ is
\begin{equation}
\label{eq:thetaid}
 \theta_q = \frac{g^2}{4(e^{2\pi i/6}-1)f}\theta_Z.
\end{equation}

Our next goal is to use these identities to transform the differential equation
\[
D(A,B) = (A,B)\twomat{0}{u f}{g}{0}
\]
from Proposition \ref{p:inductioncase1} into an ordinary differential equation on the $Z$-line. First off, we can rescale by a power of $\eta$ as usual to assume that the classical weight is $0$ and $D = \theta_q$. Therefore, we want to solve
\[
  \theta_Z(A,B) = 4(e^{2\pi i/6}-1)(A,B) \twomat{0}{u f^2/g^2}{f/g}{0}.
\]
Now the trouble is that the matrix on the right transforms according to nontrivial characters of $G$!

So assume that $A,B$ satisfy the MLDE above, and let us find the MLDE satisfied by $(f/g)A$ and $(g/f)B$, which both transform without character and have the same weights as $A$ and $B$. A straightforward computation shows that $f$ and $g$ are linearly independent solutions of the MLDE:
\begin{align*}
  D^2F &= \frac{1}{18}E_4F.
\end{align*}
We'll also need to evaluate $D(f/g)$ and $D(g/f)$, and for this we need to know $D(f)$ and $D(g)$. Since $D$ increases the weight by $2$ and leaves the character invariant, it must be the case that $D(f) = ug^2$ and $D(g) = vf^2$ for some $u$ and $v$. A straightforward computation shows that $D(f) = \frac{1}{12}e^{2\pi i/6}g^2$ and $D(g) = \frac{1}{12}e^{2\pi i/6}f^2$. Therefore,
\begin{align*}
  D(f/g) &= \frac{1}{12}e^{2\pi i/6}\frac{g^3-f^3}{g^2},&  D(g/f) &=\frac{1}{12}e^{2\pi i/6}\frac{f^3-g^3}{f^2}.
\end{align*}

Next we set $\alpha = (f/g)A$ and $\beta = (g/f)B$ then we find that
\begin{align*}
  D(\alpha,\beta) &=D\left((A,B)\twomat{f/g}{0}{0}{g/f}\right)\\
                  &= (A,B)\twomat{0}{uf}{g}{0}\twomat{f/g}{0}{0}{g/f} + (A,B)\twomat{D(f/g)}{0}{0}{D(g/f)}\\
                  &= (A,B)\twomat{0}{ug}{f}{0} + \frac{1}{18}e^{2\pi i/6}(A,B)\twomat{(g^3-f^3)/g^2}{0}{0}{(f^3-g^3)/f^2}\\
                  &= ((g/f)\alpha,(f/g)\beta)\left(\twomat{0}{ug}{f}{0} + \frac{1}{18}e^{2\pi i/6}\twomat{(g^3-f^3)/g^2}{0}{0}{(f^3-g^3)/f^2}\right)\\
  &= (\alpha,\beta)\left(\twomat{0}{ug^2/f}{f^2/g}{0} + \frac{1}{18}e^{2\pi i/6}\twomat{(g^3-f^3)/fg}{0}{0}{(f^3-g^3)/fg}\right)
\end{align*}
Therefore, we want to solve the differential equation
\begin{align*}
  \theta_Z(\alpha,\beta) &= \frac{4(e^{2\pi i/6}-1)f}{g^2}(\alpha,\beta)\left(\twomat{0}{ug^2/f}{f^2/g}{0} + \frac{1}{18}e^{2\pi i/6}\twomat{(g^3-f^3)/fg}{0}{0}{(f^3-g^3)/fg}\right)\\
  &= 4(e^{2\pi i/6}-1)(\alpha,\beta)\left(\twomat{0}{u}{f^3/g^3}{0} + \frac{1}{18}e^{2\pi i/6}\twomat{(g^3-f^3)/g^3}{0}{0}{(f^3-g^3)/g^3}\right)
\end{align*}
Everything above can be expressed in terms of $Z$. We know that $(g/f)^3 = 1-Z$, and hence $(f/g)^3 = (1-Z)^{-1}$. Thus $(g^3-f^3)/g^3 = 1-(f/g)^3 = 1-(1-Z)^{-1} = -\frac{Z}{(1-Z)}$ and $(f^3-g^3)/g^3 = \frac{Z}{(1-Z)}$. Thus, at the end of the day we find that $(\alpha,\beta)$ satisfy the matrix differential equation
\[
  \theta_Z(\alpha,\beta) = (\alpha,\beta)\twomat{\frac{2Z}{9(1-Z)}}{4e^{2\pi i/3}u}{\frac{4e^{2\pi i/3}}{1-Z}}{-\frac{2Z}{9(1-Z)}}.
\]
This is Fuchsian on the $Z$-line of degree two, and hence if we choose a cyclic vector we will be able to solve this in terms of hypergeometrics. Right multiplication of $(\alpha,\beta)$ by the matrix
\[
  P = \twomat{1}{0}{\frac{2Z}{9(1-Z)}}{\frac{4e^{2\pi i/3}}{1-Z}}
\]
transforms this system of equations into the system defined by the scalar equation
\begin{equation}
  \label{eq:fuchsian}
  \theta_Z^2F -\frac{Z}{1-Z}\theta_ZF + \frac{14Z^2-18(72e^{2\pi i/6}u+1)Z+1296e^{2\pi i/6}u}{81(1-Z)^2}F = 0.
\end{equation}
This is not hypergeometric, but it can be solved using hypergeometrics according to classical results going back to Riemann. Given a basis of solutions, we can then multiply by $P^{-1}$ to get a basis of solutions to our original equation. We deduce the following:
\begin{thm}
  \label{t:case2}
  Let $\rho$, $\rho\otimes \beta$ and $\rho\otimes \beta^2$ be as above, where $\rho$ is the restriction of a representation of $\SL_2(\ZZ)$. Let $a,b$ denote a basis of solutions to equation \eqref{eq:fuchsian}. Then there exists a choice of $\rho$ in its isomorphism class such that a nonzero minimal weight form in $M_{k_1,L}(G,\rho\otimes \beta)$ is
\[
  A = \eta^{2k_1}(g/f)\twovec{a}{b}.
\]
A nonzero minimal weight form in $M_{k_1,L}(G,\rho\otimes \beta^2)$ is
\[
  B = (e^{2\pi i/6}/36)\eta^{2k_1}(f/g)\twovec{2Za + 9(Z-1)\theta_Z(a)}{2Zb + 9(Z-1)\theta_Z(b)}.
\]
\end{thm}
\begin{proof}
  Use the computations above, noting that
  \[
  \twomat{a}{\theta_Z(a)}{b}{\theta_Z(b)}\twomat 1{\frac{e^{2\pi i/6}}{18}Z}0{\frac{e^{2\pi i/6}}{4}(Z-1)} = \twomat{a}{\frac{e^{2\pi i/6}}{18}Za + \frac{e^{2\pi i/6}}{4}(Z-1)\theta_Z(a)}{b}{\frac{e^{2\pi i/6}}{18}Zb + \frac{e^{2\pi i/6}}{4}(Z-1)\theta_Z(b)}.
\]
\end{proof}

Now momentarily let $\rho$ denote a representation of $G$ of arbitrary rank $d$, and let $L$ denote a choice of exponents for $T^2$,  so that $\rho(T^2) = e^{2\pi i L}$. We wish to describe the induced exponents $\Ind L$ satisfying $(\Ind \rho)(T) = e^{2\pi i \Ind L}$ such that
\[
  \pi_*\cV_{k}(G,\rho,L) \cong \cV_{k}(\Gamma,\Ind \rho,\Ind L)
\]
where $\pi \colon X_G \to X_{\Gamma}$ denotes the projection map beween the modular curves defined by $G$ and $\Gamma$, respectively. As for the case of tensor products and symmetric powers, it is not clear that the pushforward is a bundle of this form. We will show it is the case, and the computation comes down to examining the behaviour of the pushforward at the cusp.

It will be useful when discussing the $q$-expansion condition to use $T$ as the nontrivial coset representative for $G$ in $\Gamma$. First recall that in our basis for $\Ind \rho$ we have
\[
  (\Ind \rho)(g) = \begin{cases}
    \twomat{\rho(g)}{0}{0}{\rho(T^{-1}gT)} & g \in G,\\
    \twomat{0}{\rho(gT)}{\rho(T^{-1}g)}{0} & g \not\in G.
  \end{cases}
\]
Therefore, a holomorphic function $F \colon \uhp\to \CC^d$ satisfies
\[
  F(\gamma \tau) = ( c\tau+d)^k\rho(\gamma)F(\tau)
\]
for all $\gamma = \stwomat abcd \in G$ if and only if
\[
   \twovec{F(\gamma\tau)}{(F|T^{-1})(\gamma\tau)} = (c\tau+d)^k(\Ind\rho)(\gamma)\twovec{F(\tau)}{(F|T^{-1})(\tau)}
\]
for all $\gamma \in \Gamma$. To verify that the latter implies the former, restrict to $\gamma \in G$ and consider only the $d$ uppermost entries. For the other direction, it suffices to verify that the second identity holds when $\gamma = T$, a representative for the nontrivial coset of $G$ in $\Gamma$. In this case we can check directly:
\begin{align*}
  \twovec{F(\tau+1)}{(F|T^{-1})(\tau+1)} &= \twovec{(F|T)(\tau)}{F(\tau)}\\
  &= \twomat{0}{\rho(T^2)}{1}{0}\twovec{F(\tau)}{(F|T^{-1})(\tau)}
\end{align*}
The induced exponents $\Ind L$ should have the property that
\[
  \tilde F(\tau) = e^{-\pi i L\tau}F(\tau) = q_2^{-L}F(\tau)
\]
has a holomorphic $q_2 = e^{\pi i\tau}$-expansion if and only if
\[
\widetilde{(\Ind F)}(\tau) = e^{-2\pi i (\Ind L)\tau}\left(\begin{matrix}
    F(\tau)\\
    (F|T^{-1})(\tau)
  \end{matrix}
\right)
\]
has a holomorphic $q = e^{2\pi i\tau}$-expansion.

\begin{lem}
  \label{l:inducedexponents}
Let $\rho$ denote a finite dimensional complex representation of $G$ and let $L$ denote a choice of exponents for $\rho$, so that $L$ sastisfies $\rho(T^2) = e^{2\pi i L}$. Then the following properties hold:
\begin{enumerate}
   \item there is an isomorphism $M(G,\rho,L) \cong M(\Gamma, \Ind \rho, \Ind L)$ of graded $M(1)$-modules;
  \item in the basis for $\Ind \rho$ used above, we have
    \[  \Ind L = \frac 12\twomat{1}{e^{\pi iL}}{1}{-e^{\pi iL}}^{-1}\twomat{L}{0}{0}{L+1}\twomat{1}{e^{\pi iL}}{1}{-e^{\pi iL}};\]
  \item $\Tr(\Ind L) = \Tr(L) + \frac{\dim \rho}{2}$.
  \end{enumerate}
\end{lem}
\begin{proof}
  Part (1) is a standard result about sections of pushforward bundles (see e.g. its use in \cite{Selberg}). It was also basically established above: the isomorphism is given by $F \mapsto (F,F|T^{-1})^t$ with inverse given by projection to the first $\dim \rho$ coordinates. Part (3) follows immediately from part (2). Thus it remains to establish part (2).
 
If $F$ is modular for $(\rho,L)$, write $F(\tau) = q_2^L\sum_{n\geq 0}a_nq_2^n$. Then
\[
  (F|T^{-1})(\tau) = e^{-\pi iL}q_2^L\sum_{n\geq 0}a_n(-1)^nq_2^n.
\]
Therefore
\begin{align*}
  F(\tau)+e^{\pi iL}(F|T^{-1})(\tau)& = 2q_2^L\sum_{n\geq 0}a_{2n}q^n\\
  F(\tau)-e^{\pi iL}(F|T^{-1})(\tau)& = 2q_2^{L+1}\sum_{n\geq 0}a_{2n+1}q^n
\end{align*}
That is,
\[
  \twomat{1}{e^{\pi iL}}{1}{-e^{\pi iL}}\left(\begin{matrix}
    F(\tau)\\
    (F|T^{-1})(\tau)
  \end{matrix}
\right) = 2\twomat{q_2^L}{0}{0}{q_2^{L+1}} \sum_{n\geq 0}\twovec{a_{2n}}{a_{2n+1}}q^n
\]
This shows that the claimed expression for $\Ind L$ satisfies the required $q$-expansion condition.

It remains to verify that
\[e^{2\pi i \Ind L} = (\Ind \rho)(T) = \twomat {0}{\rho(T^2)}{1}{0}.\]
Since the exponential commutes with conjugation,
\begin{align*}
  e^{2\pi i \Ind L} &= \twomat{1}{e^{\pi iL}}{1}{-e^{\pi iL}}^{-1}\exp\left(\pi i\twomat{L}{0}{0}{L+1}\right)\twomat{1}{e^{\pi iL}}{1}{-e^{\pi iL}}\\
                    &= -\frac 12 e^{-\pi iL}\twomat{-e^{\pi i L}}{-e^{\pi i L}}{-1}{1}\twomat{e^{\pi i L}}{0}{0}{-e^{\pi i L}}\twomat{1}{e^{\pi iL}}{1}{-e^{\pi iL}}\\
                    &=\frac 12 \twomat{1}{1}{e^{-\pi iL}}{-e^{-\pi iL}}\twomat{e^{\pi i L}}{\rho(T^2)}{-e^{\pi i L}}{\rho(T^2)}\\
  &= \twomat{0}{\rho(T^2)}{1}{0}
\end{align*}
This confirms that the induced exponents are as claimed.
\end{proof}

Now take $\rho = \rho(e,\zeta_1,\zeta_2,\zeta_3,a)$ to be an irreducible representation of $G$ such that $\zeta_1+\zeta_2+\zeta_3 = 0$, so that $\rho$ is the restriction of a representation of $\SL_2(\ZZ)$ by Proposition \ref{p:irrepsofG}. For all but finitely many explicitly computable values of $a$, Proposition \ref{p:irrepsofG} also shows that $\Ind(\rho\otimes \beta)$ and $\Ind(\rho \otimes \beta^2)$ are irreducible. We can now give a formula for the minimal weights of $\rho \otimes \beta$ and $\rho \otimes \beta^2$, since they must agree with the minimal weights of the induced representations. As remarked above, we will see again that these minimal weights are equal.
\begin{lem}
  \label{l:inductionweights}
  Let $\rho$ denote an irreducible representation of $G$ of rank two that is the restriction of a representation of $\Gamma$, and assume that $\Ind( \rho\otimes \beta)$ and $\Ind( \rho \otimes \beta^2)$ are irreducible. Let $L$ denote a choice of exponents for $\rho$, $\rho\otimes \beta$ and $\rho \otimes \beta^2$, and write $\rho(-1) = (-1)^e$.
  \begin{enumerate}
    \item The minimal classical weights for $(\rho \otimes \beta,L)$ and $(\rho\otimes \beta^2,L)$ are both equal to $k_1$ where
  \[
  k_1 = \begin{cases}
    3\Tr(L) & 3\Tr(L) \not\equiv e \pmod{2},\\
    3\Tr(L) + 1 & 3\Tr(L) \equiv e \pmod{2}.
  \end{cases}
\]
\item If $k_1 = 3\Tr(L)$ then $M(\Gamma,\Ind(\rho\otimes \beta), \Ind L)$ and $M(\Gamma,\Ind(\rho\otimes \beta^2),\Ind L)$ correspond to the cyclic case of Theorem \ref{t:weights}. Otherwise they both correspond to the noncyclic case of Theorem \ref{t:weights}.
\end{enumerate}
\end{lem}
\begin{proof}
  The parities of $\Ind \rho$, $\Ind (\rho \otimes \beta)$ and $\Ind (\rho\otimes \beta)$ are equal to the parity $e$ of $\rho$. Since
\[
  M(G,\rho\otimes \beta^j,L) \cong M(\Gamma, \Ind(\rho \otimes \beta^j), \Ind L)
\]
by Lemma \ref{l:inducedexponents}, the minimal classical weight for $(\rho \otimes \beta^j,L)$ is equal to the minimal weight for $(\Ind(\rho\otimes \beta^j),\Ind_e L)$. Since $\Ind(\rho \otimes \beta^j)$ is irreducible when $j=1,2$, in these cases the minimal weight can be computed from Theorem \ref{t:weights}: since the parity of the inductions is equal to the parity of $\rho$, the minimal weights for $\Ind (\rho \otimes \beta)$ and $\Ind(\rho \otimes \beta^2)$ relative to the exponents $\Ind_e L$ are equal to
\[
  k_1 = \begin{cases}
    3\Tr(\Ind_e L)-3 & 3\Tr(\Ind_e L) \not \equiv e \pmod{2},\\
    3\Tr(\Ind_e L)-2 & 3\Tr(\Ind_3 L) \equiv e \pmod{2}.
  \end{cases}
\]
The first case above is the cyclic case, while the second case is the noncyclic case. Since $\Tr(\Ind_e L) = \Tr(L) +1$ by Lemma \ref{l:inducedexponents}, the result follows.
\end{proof}

We are now prepared to explain how to describe all modular forms for rank $4$ irreducible representations of $\Gamma$ that arise by induction from irreducible representations of $G$ of rank $2$, as long as one uses induced exponents. First normalize the orbit $\rho$, $\rho\otimes \beta$, $\rho\otimes \beta^2$ so that $\rho$ is the restriction of a representation of $\Gamma$. Assume that $\alpha_1 =\rho\otimes \beta$ and $\alpha_2 = \rho\otimes \beta^2$ are irreducible. Such representations are classified by Proposition \ref{p:irrepsofG}. Let $L$ denote a choice of exponents for $\rho(T^2)$ and set $\Lambda = \Ind L$, where $\Lambda$ is as in Lemma \ref{l:inducedexponents}. The minimal weights for the free $\Gamma$-modules $M(\Gamma,\alpha_1,\Lambda)$ and $M(\Gamma,\alpha_2,\Lambda)$ of modular forms of rank four are both equal to $k_1$ as in Lemma \ref{l:inductionweights}. By our work in Sections \ref{s:cyclic} and \ref{s:noncyclic}, in either the cyclic or noncyclic case of Theorem \ref{t:weights}, we can describe all elements of $M(\Gamma, \alpha_1,\Lambda)$ and $M(\Gamma,\alpha_2,\Lambda)$ in terms of a form of minimal weight and its modular derivatives. Thus, we have reduced our problem of describing $M(\Gamma, \alpha_1,\Lambda)$ and $M(\Gamma,\alpha_2,\Lambda)$ to the problem of giving a formula for a form of minimal weight in each space.

The key point is that by part (1) of Lemma \ref{l:inducedexponents}, we have $M(G,\rho\otimes \beta,L) \cong M(\Gamma,\alpha_1,\Lambda)$ and $M(G,\rho\otimes \beta^2,L)\cong M(\Gamma,\alpha_2,\Lambda)$. These isomorphisms take a form $F$ on $G$ to $(F,F|T^{-1})^t$ on $\Gamma$. Thus, $q$-expansions for minimal weight forms in $M(\Gamma, \alpha_1,\Lambda)$ and $M(\Gamma,\alpha_2,\Lambda)$ can be computed once the $q$-expansions for minimal weight form in $M(G, \rho\otimes \beta^j,L)$ for $j=1,2$ are known.  

Recall that Proposition \ref{p:inductioncase1} and the discussion following it explained how to determin a differential equation satisfied by the forms of minimal weight in $M(G, \rho\otimes \beta^j,L)$ for $j=1,2$. As explained in Theorem \ref{t:case2}, formulas for these minimal weight forms can be found by solving the ordinary differential equation \eqref{eq:fuchsian}. While not hypergeometric, solving equation \eqref{eq:fuchsian} can be reduced to solving a hypergeometric equation using standard techniques going back to Riemann. In this way one can recover explicit formulas for generators of $M(\Gamma,\alpha_1,\Lambda)$ and $M(\Gamma,\alpha_2,\Lambda)$. It is cumbersome to make these formulas explicit, and often for computations it is better to solve Equation \eqref{eq:fuchsian} recursively, anyway. Thus, we will summarize this discussion in a weak form:
\begin{thm}
  \label{t:maininduction}
  Let $\rho$, $\rho\otimes \beta$ and $\rho\otimes\beta^2$ be irreducible representations of $G$ normalized so that $\rho$ is the restriction of a representation of $\Gamma$. Assume that $\alpha_1 = \Ind_G^\Gamma(\rho\otimes \beta)$ and $\alpha_2 = \Ind_G^\Gamma(\rho\otimes \beta^2)$ are irreducible. Let $L$ be a choice of exponets for $\rho(T^2)$ and let $\Lambda = \Ind_eL$ be the induced exponents for $\alpha_1$ and $\alpha_2$. Then all modular forms in the free $\Gamma$-modules $M(\Gamma,\alpha_1, \Lambda)$ and $M(\Gamma,\alpha_2,\Lambda)$ can be expressed entirely in terms of products and sums of $E_4$, $E_6$, powers of the $\eta$-function, algebraic functions of
  \[
   Z =2\frac{12^{3/2} \eta^{12}}{12^{3/2}\eta^{12}+iE_6},
 \]
 and hypergeometric series in $Z$.
\end{thm}
As for our results on tensor products and symmetric cubes, Theorem \ref{t:maininduction} only describes modular forms for $\alpha_1$ and $\alpha_2$ with respect to induced exponents $\Ind L$. There are choices of exponents $L'$ that do not arise from induction, but it is always possible to choose $L$ so that there is an inclusion of lattices
\[
  M(G,\alpha_j, L') \subseteq M(G,\alpha_j,\Ind L) \subseteq M^\dagger(G,\alpha_j).
\]
Thus, Theorem 41 applies in fact to all weakly holomorphic modular forms for $\alpha_j$. If one wanted to prove unbounded denominator type results as in say \cite{FrancMason1}, then one could (and would need) to make Theorem \ref{t:maininduction} more precise. We do not recommend this, however, as the resulting computations would be quite involved.

\setcounter{tocdepth}{2}

\bibliographystyle{plain}

\end{document}